\documentclass[11pt,a4paper,reqno]{amsart}

\usepackage{graphicx,mathrsfs,amssymb}
\usepackage[usenames]{color}
\usepackage[colorlinks=true,citecolor=blue]{hyperref}
\usepackage{enumerate}

\newtheorem{theorem}{Theorem}[section]
\newtheorem{lemma}[theorem]{Lemma}
\newtheorem{corollary}[theorem]{Corollary}
\newtheorem{proposition}[theorem]{Proposition}
\theoremstyle{definition}
\newtheorem{definition}[theorem]{Definition}
\newtheorem{remark}[theorem]{Remark}
\newtheorem{example}[theorem]{Example}

\numberwithin{equation}{section}

\newcommand{\Z}{{\mathbb Z}}
\newcommand{\T}{{\mathbb T}}

\newcommand{\N}{{\mathbb N}}
\newcommand{\R}{{\mathbb R}}

\newcommand{\dbar}{d\hspace*{-0.08em}\bar{}\hspace*{0.1em}}

\newcommand{\supp}{\mathop{\mathrm{supp}}}
\newcommand{\var}{\mathop{\mathrm{Var}}}

\newcommand{\Keywords}[1]{{\footnotesize\noindent\textit{Keywords:}
  \parbox[t]{100mm}{\raggedright\footnotesize#1}}\vspace{.5pc}}
\newcommand{\ams}[2]{{\footnotesize\noindent AMS 2010 \textit{Mathematics subject
classification:} Primary #1\\[-2pt]\phantom{\footnotesize\noindent AMS 2010
\textit{Mathematics subject classification:}} Secondary #2\vspace{1pc}}}

\renewcommand{\epsilon}{\varepsilon}
\renewcommand{\tilde}[1]{\widetilde{#1}}

\allowdisplaybreaks
\begin{document}
\title[Operator-valued Fourier multipliers ]{Operator-valued Fourier multipliers on periodic Besov spaces\footnote{This work
was partially supported by COLCIENCIAS, Grant 121556933488.}}
\author{\ \ B. Barraza Mart\'inez}
\address{B. Barraza Martinez, Universidad del Norte, Departamento de Matem\'aticas,  Barranquilla (Colombia)}
\email{bbarraza@uninorte.edu.co}

\author{I. Gonz\'alez Mart\'inez}
\address{I. Gonz\'alez Mart\'inez, Cinvestav, Departamento de Matem\'aticas, Mexico D.F. (Mexico)}
\email{idgonzalez@math.cinvestav.mx}

\author{J. Hern\'andez Monz\'on}
\address{J. Hern\'andez Monz\'on, Universidad del Norte, Departamento de Matem\'aticas,  Barranquilla (Colombia)}
\email{jahernan@uninorte.edu.co}

\date{\today}

\maketitle

\begin{abstract}
We prove in this paper that a sequence $M:\mathbb{Z}^{n}\to\mathcal{L}(E)$ of bounded variation is a Fourier multiplier on the Besov space $B_{p,q}^{s}(\mathbb{T}^{n},E)$ for $s\in\R$, $1<p<\infty$, $1\leq q\leq\infty$ and $E$ a Banach space, if and only if $E$ is a
UMD-space. This extends in some sense the Theorem 4.2 in \cite{AB04} to the $n-$dimensional case. The result is used to obtain existence and uniqueness of solution for some Cauchy problems with periodic boundary conditions.\\
\end{abstract}


\ams{42A45}{47A56}

\section{Introduction}

We are interested in obtaining a Fourier multiplier theorem on the periodic
Besov spaces $B_{p,q}^{s}\left(  \mathbb{T}^{n},E\right)  $ for $s\in
\mathbb{R}$, $1<p<\infty,$ $1\leq q\leq\infty$ and $E$ a $UMD-$space, with
discrete symbols satisfaying a bounded variation condition\ similar to the
introduced in \cite{Zim89} and \cite{SW07}. Thus we can consider Cauchy problems
with periodic boundary conditions. To reach this goal we give an extension of Theorem 4.2 in \cite{AB04} and then we analyze the following two Cauchy problems:
\begin{equation}
\left\{
\begin{array}
[c]{rr}%
\partial_{t}u(t,x)+A(t)u(t,x)=f(t,x)\text{,} & t\in\left(  0,T\right]  \text{,
}x\in\mathbb{T}^{n}\text{,}\\
u(0,x)=u_{0}(x)\text{,} & x\in\mathbb{T}^{n}\text{,}%
\end{array}
\right.  \label{Prob1}%
\end{equation}
and%
\begin{equation}
\left\{
\begin{array}
[c]{rr}%
\partial_{t}u(t,x)+A_{\omega}u(t,x)=f(t,x)\text{,} & t\in\left[
0,2\pi\right]  \text{, }x\in\mathbb{T}^{n}\text{,}\\
u(0,x)=u(2\pi,x)\text{,} & x\in\mathbb{T}^{n}\text{,}%
\end{array}
\right.  \label{Prob2}%
\end{equation}
where $T>0$ and $A(t)$ in (\ref{Prob1}) is a family of uniformly normal
elliptic differential operators given by%
\begin{equation}
A(t):=\sum_{\left\vert \alpha\right\vert \leq m}a_{\alpha}\left(  t\right)
D^{\alpha}\text{.}\label{Glei differoperat}%
\end{equation}
Here $m\in\mathbb{N}$, $a_{\alpha}\in C_{b}\left(  \left[  0,\infty\right)
,\mathcal{L}\left(  E\right)  \right)  $ for $\left\vert \alpha\right\vert
\leq m$, $E$ is a UMD-space and $D_{j}:=-i\partial_{j}$. For the
Problem (\ref{Prob2}), $A_{\omega}:=\omega+A$, where $A$ is as in
(\ref{Glei differoperat}) but with constant coefficients and $\omega
\geq\omega_{0}$ with $\omega_{0}$ appropiated. We stress that in \cite{ABF08} the autor use the evolution equation (\ref{Prob1}) with $n=3$ and
$A(t)=-\Delta$ to investigate the concentration of a pool of soluble polymers in a
small cubical section $\left[  0,2\pi\right]  ^{3}$ ($\equiv\mathbb{T}^{3}$)
of a biological cell (see also \cite{Na12}).

For $E$ a real (or complex) Banach space, $1\leq p<\infty$ and  $n\in\mathbb{N}$ let $L^{p}\left(  \mathbb{R}^{n},E\right)$ and $L^{p}\left(\mathbb{T}^{n},E\right)$ be the usual Bochner space of $p$-integrable $E$-valued functions on $\mathbb{R}^{n}$ and on the $n$-dimensional torus $\mathbb{T}^{n}$ respectively.  Now, we say that a function $M:\mathbb{Z}^{n}\to\mathcal{L}(E)$, where $\mathcal{L}(E)$ is the Banach space
of bounded linear operators $T:E\to E$ endowed with the usual operator norm, is a Fourier multiplier on $L^{p}\left(\mathbb{T}^{n},E\right)$ if for each $f\in L^{p}\left(  \mathbb{T}^{n},E\right)$ there exists $g\in L^{p}\left(\mathbb{T}^{n},E\right)$ such that%
\begin{equation}
\hat{g}\left(  k\right)  =M(k)\hat{f}\left(  k\right)  \text{ for all }%
k\in\mathbb{Z}^{n}\text{,}\label{ec defi Fourier mult}%
\end{equation}
where $^{\wedge}$ denotes the Fourier transform. In the same way, we say that $M$ is a Fourier multiplier on $B_{p,q}^{s}\left(  \mathbb{T}^{n},E\right)$ if for each $f\in B_{p,q}^{s}\left(
\mathbb{T}^{n},E\right)$ there exists $g\in B_{p,q}^{s}\left(\mathbb{T}^{n},E\right)$ such that (\ref{ec defi Fourier mult}) holds.

In contrast to extensive theory on $E-$valued distributions in general and
Fourier multiplier theorems on $L^{p}\left(  \mathbb{R}^{n},E\right)  $ and
$B_{p,q}^{s}\left(  \mathbb{R}^{n},E\right)  $ (and its aplications to partial
differential equations) in particular, the contribution in literature to
$E-$valued periodic distributions is rather sparse. The classical Fourier
multiplier theorems of Marcinkiewicz and Mikhlin are extended to vector-valued
functions and operator-valued multipliers on $\mathbb{Z}^{n}$, which
satisfy certain $\mathcal{R}$-boundedness condition, in \cite{Zim89}, \cite{BK05a}, \cite{BK05b}, \cite{SW07} and \cite{Na12}, for example. More specifically, they established
Fourier multiplier theorems on $L^{p}\left(  \mathbb{T}^{n},E\right)  $ if
$1<p<\infty$, $E$ is a UMD-space and, instead uniform boundedness, a $\mathcal{R}$-boundedness condition similar to condition \eqref{generalizacion de la condicion variacional de Marcinkiewics} in this work holds. The firth results about the vector-valued
periodic Besov spaces $B_{p,q}^{s}\left(\mathbb{T}^{n},E\right)$ and
Fourier multiplier theorems on these spaces appeared in \cite{AB04} but with
$n=1$. There, in Theorem 4.2, the autors proved that each sequence
$M:\mathbb{Z}\to\mathcal{L}\left(E\right)$ satisfaying the
variational Marcinkiewicz condition is a Fourier multiplier on $B_{p,q}^{s}\left(  \mathbb{T},E\right)$ if and only if $1<p<\infty$ and $E$ is a
UMD-space. The corresponding result of this theorem for Besov spaces on the
real line has been established by Bu and Kim in \cite{BK05b}. The variational
Marcinkiewicz condition, giving in \cite{AB04} is equivalent to the bounded variation condition
\eqref{generalizacion de la condicion variacional de Marcinkiewics} in case $n=1$.

In this paper we obtain (Theorem \ref{RESULTADO-PRINCIPAL}) an analogous result to the assertion of Theorem 4.2 in \cite{AB04} for the periodic Besov space $B_{p,q}^{s}\left(\mathbb{T}^{n},E\right)$. Indeed we prove that, given $s\in\R$, $1<p<\infty$ and $1\leq q\leq\infty$, each fuction $M:\Z^n\to\mathcal{L}(E)$ which satisfies \eqref{generalizacion de la condicion variacional de Marcinkiewics} is a Fourier multiplier in $B_{p,q}^{s}\left(\mathbb{T}^{n},E\right)$ iff $E$ is a $UMD-$space. Of
course, the proof of the implication where the UMD character of $E$ is the thesis, is similar to the case $n=1$ in \cite{AB04}. The hard part of this work was to show the other
direction of the equivalence. With this, we establish results of existence and uniqueness of solution for the problems \eqref{Prob1} and \eqref{Prob2}, since we prove that the sequences
\[
M_{t,\lambda}(k):=\lambda\left(  \lambda+a(t,k)\right)^{-1},\text{   }%
k\in\mathbb{Z}^{n},
\]
are of bounded variation, where $a(t,\cdot)$ is the symbol of $A(t)$.

The plan of the paper is as follows: After some preliminary definitions and
remarks in Section 2, we develope in Section 3 some fundamental elements on
Besov spaces. In particular it is proved in Lemma \ref{lema04.1} the existence of a
resolution of the unity (very useful in the following sections) and in Theorem \ref{teorema-ind-res-en-besov} the independence of norms on the resolution of unity in the space $B_{p,q}^s(\T^n,E)$. In Section 4 we define discrete Fourier multipliers on $L^p(\T,E)$ and $B_{p,q}^s(\T^n,E)$, the UMD-spaces and  show in Corollary \ref{Teorema1-caracterizacion-UMD} an elementary result to characterize UMD-spaces.
We prove in Section 5 the main result of the present paper, the Theorem \ref{RESULTADO-PRINCIPAL}. As an application, we prove in Section 6 (Corollary \ref{Corollary-application}) the existence and uniqueness of solution for the problems \eqref{Prob1} and \eqref{Prob2} in certain periodic Besov spaces.\\

In the next three sections we explain in detail definitions and preliminary results for this work in order to do more understandable the study of the periodic Besov spaces $B_{p,q}^{s}\left(  \mathbb{T}^{n},E\right)$ and the main result.

\section{Functions and distributions on $\mathbb{T}^{n}$ and $\mathbb{Z}^{n}$}

In this section we will present some notations, function spaces  on the torus
$\mathbb{T}^{n}$ and on the lattice $\mathbb{Z}^n$, as well as  spaces of periodic  and tempered distributions. Furthermore we will
give some results, which are proven in similar way to the
one-dimensional case discussed in \cite{AB04} (see also \cite{BDHN} and \cite{Na12}).

Throughout this paper $n\in\mathbb{N}$ is fixed, $\mathbb{N}_{0}:=\mathbb{N\cup
}\left\{  0\right\}  $, $\left\{  \delta_{j}:j=1,...,n\right\}  $ is the
standard basis of $\mathbb{R}^{n}$, $\langle x\rangle:=(1+|x|)^{1/2}$ for
$x\in\mathbb{R}^{n}$, where $|x|$ is the euclidean norm of $x$, $B_{r}(a)$ and
$\bar{B}_{r}(a)$ denote the open ball and closed ball (respectively) of radius
$r>0$ centered at a point $a\in\mathbb{R}^{n}$. $E$ denotes an arbitary Banach
space with norm $\left\Vert \cdot\right\Vert _{_{{E}}}$. If $X$ and $Y$ are local convex spaces, then $\mathcal{L}(X,Y)$ denotes the space of all linear and
continuous applications from $X$ into $Y$. As usual $\mathcal{L}(X):=\mathcal{L}(X,X)$. For $\alpha,\beta\in\mathbb{Z}^{n}$ the writing
$\alpha\leq\beta$ means that $\alpha_{i}\leq\beta_{i}\ $for each
$i=1,\ldots,n$, and $[\alpha,\beta]:=\left\{  k\in\mathbb{Z}^{n}:\alpha\leq
k\leq\beta\right\}  $. In the following $\dbar x:=(2\pi)^{-n}dx$, where $dx$ is the Lebesgue measure. Furthermore $C_c^m(\R^n,E)$, for $m\in\N_0\cup\{\infty\}$, denote as usual the set of all $m$-times continuously differentible functions $\varphi:\R^n\to E$ with compact support and we write $C_c:= C_c^0$.

\begin{definition}
[The spaces $C^{\infty}(\mathbb{T}^{n},E)$]\label{def-CmTn} We denote with
$C^{m}(\mathbb{T}^{n},E)$, $m\in\N_0$, the space of all $2\pi$-periodic (in
each component), $E$-valued and $m$-times continuously differentiable
functions defined in $\mathbb{R}^{n}$. The space of test functions is the
space $C^{\infty}(\mathbb{T}^{n},E):=\bigcap\limits_{m\in\mathbb{N}_{0}}%
C^{m}(\mathbb{T}^{n},E)$.\newline The topology of $C^{\infty}(\mathbb{T}%
^{n},E)$ is induced by the contable family of seminorms {$\{q_{k}%
\,;\,k\in\mathbb{N}_{0}\}$} given by
\begin{equation}
q_{k}(\varphi):=\max_{%
\genfrac{}{}{0pt}{}{\alpha\in\mathbb{N}_{0}^{n}}{|\alpha|\leq k}%
}\sup_{x\in\lbrack0,2\pi]^{n}}\Vert\partial^{\alpha}\varphi(x)\Vert
,{\quad\varphi\in C^{\infty}(\mathbb{T}^{n},E)}. \label{eq-seminorma-q-CT}%
\end{equation}
It can be shown that $\left(  C^{\infty}(\mathbb{T}^{n},E),\{q_{k}%
\,;\,k\in\mathbb{N}_{0}\}\right)  $ is a Frechet space.
\end{definition}

\begin{definition}
[The space of periodic distributions $\mathcal{D}^{\prime}(\mathbb{T}^{n},E)$%
]\label{def-dist-Tn} The space\newline$\mathcal{D}^{\prime}(\mathbb{T}%
^{n},E):=\mathcal{L}(C^{\infty}(\mathbb{T}^{n}),E)$ is called the space of
$E$-valued periodic (or toroidal) distributions. The value of $u\in
\mathcal{D}^{\prime}(\mathbb{T}^{n},E)$ on a test function $\varphi\in
C^{\infty}(\mathbb{T}^{n})$ will be denoted by $u(\varphi)$ or $\langle
u,\varphi\rangle$.\newline The topology of $\mathcal{D}^{\prime}%
(\mathbb{T}^{n},E)$ is the weak-$^\ast$-topology, i.e. a sequence
$(u_{k})_{k\in\mathbb{N}}$ in $\mathcal{D}^{\prime}(\mathbb{T}^{n},E)$
converges to $u\in\mathcal{D}^{\prime}(\mathbb{T}^{n},E)$ iff,
\[
\langle u_{k},\varphi\rangle\underset{k\rightarrow\infty}{\longrightarrow
}\langle u,\varphi\rangle\quad\text{in}\ E,\quad\forall\,\varphi\in C^{\infty
}(\mathbb{T}^{n}).
\]
{Indeed the topology of $\mathcal{D}^{\prime}(\mathbb{T}^{n},E)$ is inducided
by the family of seminorms $\{q_{\varphi}^{\prime}\,;\,\varphi\in C^{\infty
}(\mathbb{T}^{n})\}$ where
\[
q_{\varphi}^{\prime}(u):=\Vert u(\varphi)\Vert,\quad u\in\mathcal{D}^{\prime
}(\mathbb{T}^{n},E),\quad\varphi\in C^{\infty}(\mathbb{T}^{n}).
\]
}
\end{definition}

For example, for any $\psi\in C^{\infty}(\mathbb{T}^{n},E)$, the map
\[
C^{\infty}(\mathbb{T}^{n})\ni\varphi\mapsto\int\limits_{[0,2\pi]^{n}%
}\varphi(x)\psi(x)\,\dbar x,
\]
defines a $E-$ valued periodic distribution, which we call again $\psi$.

\begin{definition}
We denote with $L^{p}(\mathbb{T}^{n},E)$, $1\leq p\leq\infty$, the space of
all strongly measurable $2\pi$-periodic (in each component) functions $f:\mathbb{R}^{n}\to E$,  
such that $\left\Vert f\right\Vert
_{{L}^{p}{(\mathbb{T}^{n},E)}}<\infty$, where
\[
\left\Vert f\right\Vert _{{L}^{p}{(\mathbb{T}^{n},E)}}:=\left(  \int
_{[0,2\pi]^{n}}\left\Vert f(x)\right\Vert _{{E}}^{p}d\hspace*{-0.08em}\bar
{}\hspace*{0.1em}x\right)  ^{1/p},\qquad1\leq p<\infty,
\]
and with the usual definition for $p=\infty$.
\end{definition}

As the continuous case it holds%
\[
L^{p}(\mathbb{T}^{n},E)\hookrightarrow\mathcal{D}^{\prime}(\mathbb{T}%
^{n},E),\qquad\forall\,1\leq p\leq\infty.
\]

\begin{definition}
The space $\mathcal{S}(\mathbb{Z}^{n},E)$ consists of all functions
$\varphi:\mathbb{Z}^{n}\longrightarrow E$ for which the following holds: For
each $M\in\mathbb{R}$ there exists a constant $C_{\varphi,M}$ such that
\begin{equation}
\left\Vert \varphi(\xi)\right\Vert _{E}\leq C_{\varphi,M}\langle\xi
\rangle^{-M},\qquad\text{for all }\xi\in\mathbb{Z}^{n}\text{.}%
\end{equation}
The elements of $\mathcal{S}(\mathbb{Z}^{n},E)$ are called $E-$valued rapidly
decreasing functions on $\mathbb{Z}^{n}$. As usual $\mathcal{S}(\mathbb{Z}%
^{n}):=\mathcal{S}(\mathbb{Z}^{n},\mathbb{C})$.
\end{definition}

The topology in $\mathcal{S}(\mathbb{Z}^{n},E)$ is given by the contable
family of seminorms $\left\{  p_{k}:k\in\mathbb{N}_{0}\right\}  $ defined by
\begin{equation}
p_{k}(\varphi):=\sup_{\xi\in\mathbb{Z}^{n}}\langle\xi\rangle^{k}\left\Vert
\varphi(\xi)\right\Vert _{E},\qquad\text{for }\varphi\in\mathcal{S}%
(\mathbb{Z}^{n},E).
\end{equation}
Then a secuence $\left(  \varphi_{l}\right)  _{l\in\mathbb{N}}$ in
$\mathcal{S}(\mathbb{Z}^{n},E)$ converges to a function $\varphi\in\mathcal{S}%
(\mathbb{Z}^{n},E)$ iff
\[
p_{k}\left(  \varphi_{l}-\varphi\right)  \underset{l\rightarrow\infty
}{\longrightarrow}0\text{ for all }k\in\mathbb{N}_{0}.
\]

The space of $E-$valued tempered distributions on $\mathbb{Z}^{n}$ will be
denoted by $\mathcal{S}^{\prime}(\mathbb{Z}^{n},E)$ and consists of all linear
and continuous mappings from $\mathcal{S}(\mathbb{Z}^{n})$ into $E$. This distributions space is also endowed with the weak-$^\ast$-topology. 

\begin{example}
\label{Exampl multip_por_distrib} Let $\phi\in C_{c}(\mathbb{R}^{n})$ and
$f\in\mathcal{S}^{\prime}(\mathbb{Z}^{n},E)$. Then the mapping $\phi
f:\mathcal{S}(\mathbb{Z}^{n})\longrightarrow E$ defining by $(\phi
f)(\varphi):=f(\phi\varphi)$ for all $\varphi\in\mathcal{S}(\mathbb{Z}^{n})$
belongs to $\mathcal{S}^{\prime}(\mathbb{Z}^{n},E)$.
\end{example}

\begin{definition}
$a)$ For a function $f\in C^{\infty}(\mathbb{T}^{n},E)$ we define
\begin{equation}
\left(  \mathcal{F}_{\mathbb{T}^{n}}f\right)  (\xi):=\int\limits_{\mathbb{T}%
^{n}}e^{-ix\cdot\xi}f(x)d\hspace*{-0.08em}\bar{}\hspace*{0.1em}x=\int
\limits_{[0,2\pi]^{n}}e^{-ix\cdot\xi}f(x)d\hspace*{-0.08em}\bar{}%
\hspace*{0.1em}x,\qquad\xi\in\mathbb{Z}^{n}\text{.}%
\end{equation}
We call $\mathcal{F}_{\mathbb{T}^{n}}f$ the toroidal or periodic Fourier
transform of $f$.

\begin{itemize}
\item[$b)$] For $g\in\mathcal{S}(\mathbb{Z}^{n},E)$ we define%
\begin{equation}
\left(  \mathcal{F}_{\mathbb{T}^{n}}^{-1}g\right)  (x):=\sum\limits_{\xi
\in\mathbb{Z}^{n}}e^{ix\cdot\xi}g(\xi),\quad x\in\mathbb{T}^{n}\text{.}%
\end{equation}
We call $\mathcal{F}_{\mathbb{T}^{n}}^{-1}g$ the inverse periodic Fourier
transform of $g$.

\item[$c)$] Let $u\in\mathcal{D}^{\prime}(\mathbb{T}^{n},E)$. The periodic
Fourier transform of $u$ is defined by
\begin{equation}
\left(  \mathcal{F}_{\mathbb{T}^{n}}u\right)  (\varphi):=u\left(
[\mathcal{F}_{\mathbb{T}^{n}}^{-1}\varphi](-\cdot)\right)  ,\qquad\varphi
\in\mathcal{S}(\mathbb{Z}^{n}).
\end{equation}

\item[$d)$] For $v\in\mathcal{S}^{\prime}(\mathbb{Z}^{n},E)$ we define the
inverse \textit{periodic Fourier transform of }$v$ by
\begin{equation}
\left(  \mathcal{F}_{\mathbb{T}^{n}}^{-1}v\right)  (\psi):=v\left(
[\mathcal{F}_{\mathbb{T}^{n}}\psi](-\cdot)\right)  ,\qquad\psi\in C^{\infty
}(\mathbb{T}^{n})\text{.}%
\end{equation}

\end{itemize}
\end{definition}

\begin{proposition}
The following mappings are linear and continuous: $a)$ $C^{\infty}%
(\mathbb{T}^{n},E)\ni f\mapsto\mathcal{F}_{\mathbb{T}^{n}}f\in\mathcal{S}%
(\mathbb{Z}^{n},E)$, $b)$ $\mathcal{S}(\mathbb{Z}^{n},E)\ni g\mapsto
\mathcal{F}_{\mathbb{T}^{n}}^{-1}g\in C^{\infty}(\mathbb{T}^{n},E)$, $c)$
$\mathcal{D}^{\prime}(\mathbb{T}^{n},E)\ni u\mapsto\mathcal{F}_{\mathbb{T}%
^{n}}u\in\mathcal{S}^{\prime}(\mathbb{Z}^{n},E)$ and $d)$ $\mathcal{S}%
^{\prime}(\mathbb{Z}^{n},E)\ni v\mapsto\mathcal{F}_{\mathbb{T}^{n}}^{-1}%
v\in\mathcal{D}^{\prime}(\mathbb{T}^{n},E).$
\end{proposition}

\begin{definition}
We say that a function $u:\mathbb{Z}^{n}\longrightarrow E$ grows at most
polynomially at infinity if there exist constans $M\in\mathbb{R}$ and $C\geq0$
(both depending on $u$) such that
\begin{equation}
\left\Vert u(\xi)\right\Vert _{E}\leq C\langle\xi\rangle^{M},\qquad\text{for
all }\xi\in\mathbb{Z}^{n}.
\end{equation}

\end{definition}

The space of all functions $u:\mathbb{Z}^{n}\longrightarrow E$ with at most
polynomial growth at infinity will be denoted by $\mathcal{O}(\mathbb{Z}^{n},E)$.\\

Note that if $u\in\mathcal{S}(\mathbb{Z}^{n},E)$, then $u\in\mathcal{O}(\mathbb{Z}^{n},E)$. We can also identify the space 
$\mathcal{O}(\mathbb{Z}^{n},E)$ with the space of all sequences $\left(  a_{k}\right)
_{k\in\mathbb{Z}^{n}}$ in $E$, for which there are constants $C$ and $M$ such
that $\left\Vert a_{k}\right\Vert _{_{{E}}}\leq C\langle k\rangle^{M}$ for all
$k\in\mathbb{Z}^{n}$.

\begin{example}
\label{ejemplo1.12-barraza}Let $u\in\mathcal{S}^{\prime}(\mathbb{Z}^{n},E)$.
The function defined by $\bar{u}(\xi):=u(\psi_{\xi})$, $\xi\in\mathbb{Z}^{n}$
belongs to $\mathcal{O}(\mathbb{Z}^{n},E)$, where $\psi_{\xi}%
\in\mathcal{S}(\mathbb{Z}^{n})$ is defined by
\begin{equation}
\psi_{\xi}(k):=\delta_{\xi,k}:=%
\begin{cases}
1, & \text{if}\ \xi=k,\\
0, & \text{if}\ \xi\neq k.
\end{cases}
\label{eje1.7bienb}%
\end{equation}

\end{example}

\begin{proposition}
The map $\mathcal{O}(\mathbb{Z}^{n},E)\ni u\mapsto\Lambda_{u}%
\in\mathcal{S}^{\prime}(\mathbb{Z}^{n},E)$, where
\begin{equation}
\Lambda_{u}(\varphi):=\sum_{\xi\in\mathbb{Z}^{n}}\varphi(\xi
)u(\xi),\qquad\forall\varphi\in\mathcal{S}(\mathbb{Z}^{n})\text{,}%
\end{equation}
is bijective.
\end{proposition}

\begin{remark}
\label{obs1.15-barraza} Due to the last proposition, one can identify
$u\in\mathcal{S}^{\prime}(\mathbb{Z}^{n},E)$ with $\Lambda_{\bar{u}}$, and so
\begin{equation}
u(\varphi)=\Lambda_{\bar{u}}(\varphi)=\sum_{\xi\in\mathbb{Z}^{n}}\varphi(\xi)u(\psi_{\xi
}),\qquad\forall\varphi\in\mathcal{S}(\mathbb{Z}^{n}).
\end{equation}

\end{remark}

\begin{definition}
For $\phi\in C^{\infty}(\mathbb{T}^{n}),$ $u\in\mathcal{D}^{\prime}%
(\mathbb{T}^{n})$ and $e\in E$, the tensor products $\phi\otimes e$ and
$u\otimes e$ are defined by%
\[
(\phi\otimes e)(x):=\phi(x)e,\quad x\in\lbrack0,2\pi]^{n},
\]%
\[
(u\otimes e)(\varphi):=u(\varphi)e,\quad\varphi\in C^{\infty}(\mathbb{T}%
^{n}).
\]

\end{definition}

It is straightforward to prove that $\phi\otimes e\in C^{\infty}(\mathbb{T}%
^{n},E)$ and $u\otimes e\in\mathcal{D}^{\prime}(\mathbb{T}^{n},E)$.

\begin{proposition}
\label{proposicion2,12-13-barraza} Let $(a_{k})_{k\in\mathbb{Z}^{n}}\subset E
$ be a sequence with at most polynomial growth at infinity (i.e. $[k\mapsto
a_{k}]\in\mathcal{O}(\mathbb{Z}^{n},E)$), then the mapping $g:C^{\infty
}(\mathbb{T}^{n})\to E$, defined by
\begin{equation}
g(\varphi):=\sum_{k\in\mathbb{Z}^{n}}\left(  \mathcal{F}_{\mathbb{T}^{n}%
}\varphi\right)  (k)a_{k},\qquad\varphi\in C^{\infty}(\mathbb{T}^{n})\text{,}
\label{bien}%
\end{equation}
belongs to $\mathcal{D}^{\prime}(\mathbb{T}^{n},E)$. Furthermore, for all
$g\in\mathcal{D}^{\prime}(\mathbb{T}^{n},E)$ it holds
\begin{equation}
g(\varphi)=\sum_{k\in\mathbb{Z}^{n}}\left(  \mathcal{F}_{\mathbb{T}^{n}%
}\varphi\right)  (k)g(e_{k})\qquad (\varphi\in C^{\infty}(\mathbb{T}%
^{n}))\text{,} \label{obs2.15 bienvenido0}%
\end{equation}
where $e_{k}(x):=e^{ik\cdot x}$ for all $x\in\mathbb{R}^{n}$.
\end{proposition}

It follows from the last proposition that for all $g\in\mathcal{D}^{\prime
}(\mathbb{T}^{n},E)$,%
\begin{equation}
g=\sum_{k\in\mathbb{Z}^{n}}e_{k}\otimes\hat{g}(k)\quad\text{in}\ \mathcal{D}%
^{\prime}(\mathbb{T}^{n},E), \label{obs2.15 bienvenido}%
\end{equation}
where $\hat{g}(k):=g(e_{-k}),\ k\in\mathbb{Z}^{n},$ are call
the\textit{\ Fourier coefficients}%
\index{coeficientes de Fourier}
of $g$. To see this fact note that $e_{k}(\varphi)=\displaystyle\int\limits_{\mathbb{T}%
^{n}}e^{ix\cdot k}\varphi(x)\,\dbar x=\left(
\mathcal{F}_{\mathbb{T}^{n}}\varphi\right)  (-k)$ for all $\varphi\in
C^{\infty}(\mathbb{T}^{n})$.\\ 
\\
We finish this section with some results, which we will need for the following one. Before, note that%
\[
e_{k}(e_{-\xi})=\int\limits_{[0,2\pi]^{n}}e^{i(k-\xi)\cdot x}\,\dbar x=\delta_{\xi,k},\qquad\text{for all }%
\xi,k\in\mathbb{Z}^{n}.
\]

\begin{lemma}
\label{a1} If $(a_{k})_{k\in\mathbb{Z}^{n}}\in\mathcal{O}(\mathbb{Z}^{n},E)$, then $\sum_{k\in\mathbb{Z}^{n}}e_{k}\otimes a_{k}\in\mathcal{D}%
^{\prime}(\mathbb{T}^{n},E)$. Furthermore, if $(b_{k})_{k\in\mathbb{Z}^{n}}\in\mathcal{O}(\mathbb{Z}^{n},E)$,
\begin{equation}
\sum_{k\in\mathbb{Z}^{n}}e_{k}\otimes a_{k}=\sum_{k\in\mathbb{Z}^{n}}%
e_{k}\otimes b_{k}\ \text{in }\mathcal{D}^{\prime}(\mathbb{T}^{n}%
,E)\Longleftrightarrow a_{k}=b_{k}\text{ for each }k\in\mathbb{Z}^{n}.
\label{igu}%
\end{equation}

\end{lemma}

\begin{proof}
Due to $e_{-k}(\varphi)=\left(  \mathcal{F}_{\mathbb{T}^{n}}\varphi\right)
(k)$ for all $\varphi\in C^{\infty}(\mathbb{T}^{n})$, it follows from
Proposition \ref{proposicion2,12-13-barraza} that the mapping
\[
\varphi\mapsto\sum_{k\in\mathbb{Z}^{n}}\left(  \mathcal{F}_{\mathbb{T}^{n}%
}\varphi\right)  (k)a_{-k}=\sum_{k\in\mathbb{Z}^{n}}e_{k}(\varphi)a_{k}%
\]
belongs to $\mathcal{D}^{\prime}(\mathbb{T}^{n},E)$, i.e. there exists some
$g\in\mathcal{D}^{\prime}(\mathbb{T}^{n},E)$ such that
\[
g=\sum_{k\in\mathbb{Z}^{n}}e_{k}\otimes a_{k}.
\]
Now, if $\sum_{k\in\mathbb{Z}^{n}}e_{k}\otimes a_{k}=\sum_{k\in\mathbb{Z}^{n}%
}e_{k}\otimes b_{k}$, then it holds for each $\xi\in\mathbb{Z}^{n}$ that
\[
a_{\xi}=\Big(  \sum_{k\in\mathbb{Z}^{n}}e_{k}\otimes a_{k}\Big)  (e_{-\xi
})=\Big(  \sum_{k\in\mathbb{Z}^{n}}e_{k}\otimes b_{k}\Big)  (e_{-\xi
})=b_{\xi}.
\]
The reciprocal is trivial.
\end{proof}

As a direct consequence of the last lemma and the equalities
(\ref{obs2.15 bienvenido0}) and (\ref{obs2.15 bienvenido}) we have that:

\begin{theorem}
\label{lema igualdad de distribuc} Let $f,g\in\mathcal{D}^{\prime}%
(\mathbb{T}^{n},E)$ and $(a_{k})_{k\in\mathbb{Z}^{n}}\in\mathcal{O}(\mathbb{Z}^{n},E)$.

\begin{itemize}
\item[a)] $f=g\Longleftrightarrow\hat{f}(k)=\hat{g}(k)$ for all
$k\in\mathbb{Z}^{n}$.

\item[b)] $f=\sum_{k\in\mathbb{Z}^{n}}e_{k}\otimes a_{k}\Longleftrightarrow
\hat{f}(k)=a_{k}$ for all $k\in\mathbb{Z}^{n}$.
\end{itemize}
\end{theorem}
\section{The periodic Besov spaces $B_{p,q}^{s}\left(  \mathbb{T}%
^{n},E\right)  $}

A sequence $\phi:=(\phi_j)_{j\in\mathbb{N}_{0}}\subset\mathcal{S}(\mathbb{R}%
^{n})$ is called a \textit{resolution of unity}%
\index{resoluci{\'{o}}n de la unidad}%
, denoted $(\phi_j)_{j\in\mathbb{N}_{0}}\in\Phi(\mathbb{R}^{n})$, if it
satisfies the following three conditions:

\begin{enumerate}
\item $\supp(\phi_{0})\subset\Omega_{0}:=\overline{B}_{{2}}(0)$ and%
\begin{equation}
\supp(\phi_{j})\subset\Omega_{j}:=\left\{  x\in\mathbb{R}^{n}%
:2^{j-1}\leq|x|\leq2^{j+1}\right\}  ,\quad j\in\mathbb{N}.
\label{ec defi Omega j}%
\end{equation}

\item $\sum\limits_{j\geq0}\phi_{j}(\xi)=1$ for all $\xi\in\mathbb{R}^{n}$.

\item \label{resol unidad 3} For each $\alpha\in\mathbb{N}_{0}^{n}$ there
exists a constant $C_{\alpha}>0$ such that
\[
\left\vert (\partial^{\alpha}\phi_{j})(\xi)\right\vert \leq C_{\alpha
}2^{-j|\alpha|}\chi_{\Omega_{j}}(\xi),\ \text{for all }\xi\in\mathbb{R}%
^{n}\ \text{and }j\in\mathbb{N}_{0}\text{,}%
\]
where $\chi_{\Omega_{j}}$ denotes the characteristic function on $\Omega_{j}$.
\end{enumerate}

The set $\Phi(\mathbb{R}^{n})$ is not empty as it is shown in the following
generalization of Lemma 4.1 in \cite{AB04}.

\begin{lemma}
\label{lema04.1} There exists a sequence $(\phi_{j})_{j\in\mathbb{N}_{0}}%
\in\Phi(\mathbb{R}^{n})$ such that

\begin{itemize}
\item[a)] $\phi_{j}\geq0$ for all $j\in\mathbb{N}_{0}$,

\item[b)] $\supp(\phi_{j})\subsetneq\Omega_{j}$ for all $j\in\mathbb{N}_{0}$,

\item[c)] $\phi_{j}(\xi)=1$ if $|\xi|\in\lbrack7\cdot2^{j-3},3\cdot2^{j-1}]$
and $j\geq3$,

\item[d)] $|\xi|\in\lbrack7\cdot2^{j-3},3\cdot2^{j-1}]$ and $j\geq3$ implies
$\xi\notin\supp(\phi_{j-1})\cap\, \supp(\phi_{j+1})$.
\end{itemize}
\end{lemma}

\begin{proof}
Let $\varphi_{0}\in\mathcal{S}(\mathbb{R}^{n})$ with $0\leq\varphi_{0}\leq1$,
$\supp(\varphi_{0})\subsetneq\overline{B}_{{2}}(0)$ and $\varphi_{0}(\xi)=1$, if
$|\xi|\leq\frac{13}{8}$. Let $\varphi$ be another function in $\mathcal{S}%
(\mathbb{R}^{n})$ such that $0\leq\varphi\leq1$, $\supp(\varphi)\subsetneq
\left\{  x\in\mathbb{R}^{n}:\frac{3}{2}\leq|x|\leq\frac{7}{2}\right\}  $ and
$\varphi(\xi)=1$, if $\frac{13}{8}\leq|\xi|\leq\frac{13}{4}$. Now, for
$j\in\mathbb{N}$ and $\xi\in\mathbb{R}^{n}$, define
\[
\varphi_{j}(\xi):=\varphi\left(  \frac{\xi}{2^{j-1}}\right)  \qquad
\text{and}\qquad\Psi(\xi):=\sum_{k=0}^{\infty}\varphi_{k}(\xi)\text{.}%
\]
Then it is clear that for $j=1,2,\dots$ it holds
\begin{align}
\supp(\varphi_{j})\subset K_{j} &  :=\left\{  \xi\in\mathbb{R}^{n}%
:3\cdot2^{j-2}\leq|\xi|\leq7\cdot2^{j-2}\right\}  ,\\
\varphi_{j}(\xi) &  =1,\ \text{if }|\xi|\in\lbrack13\!\cdot\!2^{j-4}%
,13\!\cdot\!2^{j-3}].\label{mono4.15}%
\end{align}
Due to $K_{j}\cap K_{j+2}=\emptyset$ and $\varphi_{j}(13\cdot2^{j-3}%
)=1=\varphi_{j+1}(13\cdot2^{j-3})$ for each $j\in\mathbb{N}_{0}$, we have that
\begin{equation}
\supp\left(  \varphi_{j}\right)  \cap\supp\left(  \varphi
_{j+1}\right)  \neq\emptyset\quad\text{and}\quad\supp\left(  \varphi
_{j}\right)  \cap\supp\left(  \varphi_{j+2}\right)  =\emptyset,\label{K}%
\end{equation}
for all $j\in\mathbb{N}_{0}$.\\

\textit{Assertion}: For each $\xi\in\mathbb{R}^{n}$, there exists some
$j\in\mathbb{N}_{0}$ such that%
\begin{equation}
\Psi(\xi)=\sum_{\ell=-1}^{1}\varphi_{j+\ell}(\xi)\geq1\text{,}\label{serie-finita}%
\end{equation}
where $\varphi_{-1}:=0$. In fact: Let $\xi\in\mathbb{R}^{n}$. Because
$\mathbb{R}_{0}^{+}=[0,13/8]\cup\bigcup_{j=1}^{\infty}[13\cdot2^{j-4}%
,13\cdot2^{j-3}]$, then $|\xi|\leq13/8$ {or} $13\cdot2^{j-4}\leq|\xi
|\leq13\cdot2^{j-3}$ for some $j\in\mathbb{N}$. From this, (\ref{mono4.15})
and (\ref{K}) it follows clearly (\ref{serie-finita}).\newline Now, define for
$j=0,1,2,\dots$
\begin{equation}
\phi_{j}(\xi):=\frac{\varphi_{j}(\xi)}{\Psi(\xi)},\qquad\text{for all }\xi
\in\mathbb{R}^{n}\text{.}\label{def resolution of unity in lemma}%
\end{equation}
By direct calculation we get that $(\phi)_{j\in\mathbb{N}_{0}}\in
\Phi(\mathbb{R}^{n})$ and it satisfies $a)-d)$.
\end{proof}

In the continuous case (in $\mathbb{R}^{n}$) it is well known that if
$(\phi_{j})\in\Phi(\mathbb{R}^{n})$, then there exists a constant $C_{n}>0$
such that
\begin{equation}
\left\Vert \mathcal{F}_{\mathbb{R}^{n}}^{-1}\phi_{j}\right\Vert _{L^{1}%
(\mathbb{R}^{n})}\leq C_{n},\quad\text{for all }j\in\mathbb{N}_{0},\label{ec1}%
\end{equation}
where $\mathcal{F}_{\mathbb{R}^{n}}^{-1}\phi_{j}$ is the inverse Fourier
transform (in $\mathbb{R}^{n}$) of $\phi_{j}$.

\begin{theorem}
\label{obs-teor-just} Let $\phi\in C_{c}(\mathbb{R}^{n})$ and $f\in
\mathcal{D}^{\prime}(\mathbb{T}^{n},E)$. Then
\begin{equation}
\mathcal{F}_{\mathbb{T}^{n}}^{-1}(\phi\mathcal{F}_{\mathbb{T}^{n}}%
f)=\sum_{k\in\mathbb{Z}^{n}}e_{k}\otimes\phi(k)\hat{f}%
(k).\label{mi resultado 1}%
\end{equation}

\end{theorem}

\begin{proof}
From Example \ref{Exampl multip_por_distrib} and Remark \ref{obs1.15-barraza}
it follows for $\phi\in C_{c}(\mathbb{R}^{n})$ and $f\in\mathcal{D}^{\prime
}(\mathbb{T}^{n},E)$ that
\begin{align*}
\big[\mathcal{F}_{\mathbb{T}^{n}}^{-1}(\phi\mathcal{F}_{\mathbb{T}^{n}%
}f)\big]\,^{\widehat{}}\,(k) &  =\big(\mathcal{F}_{\mathbb{T}^{n}}^{-1}%
(\phi\mathcal{F}_{\mathbb{T}^{n}}f)\big)(e_{-k})=(\phi\mathcal{F}%
_{\mathbb{T}^{n}}f)((\mathcal{F}_{\mathbb{T}^{n}}e_{-k})(-\cdot))\\
&  =\sum_{\xi\in\mathbb{Z}^{n}}\phi(\xi)(\mathcal{F}_{\mathbb{T}^{n}}%
f)(\psi_{\xi})(\mathcal{F}_{\mathbb{T}^{n}}e_{-k})(-\xi)\\
&  =\sum_{\xi\in\mathbb{Z}^{n}}\phi(\xi)(\mathcal{F}_{\mathbb{T}^{n}}%
f)(\psi_{\xi})\delta_{\xi,k}=\phi(k)(\mathcal{F}_{\mathbb{T}^{n}}f)(\psi
_{k})\\
&  =\phi(k)f\big([\mathcal{F}_{\mathbb{T}^{n}}^{-1}\psi_{k}](-\cdot
)\big)=\phi(k)f(e_{-k})=\phi(k)\hat{f}(k).
\end{align*}
That is,
\begin{equation}
\big[\mathcal{F}_{\mathbb{T}^{n}}^{-1}(\phi\mathcal{F}_{\mathbb{T}^{n}%
}f)\big]\,^{\widehat{}}\,(k)=\phi(k)\hat{f}(k)\quad\text{for all }%
k\in\mathbb{Z}^{n}\text{,}\label{mi resultado 01}%
\end{equation}
and therefore (\ref{mi resultado 1}) holds, due to (\ref{obs2.15 bienvenido}).
\end{proof}

In similar way to the proof of Proposition 2.2 in \cite{AB04} one obtaints the
following result.

\begin{proposition}
\label{teor4.13} For each $\phi\in C_{c}^{\infty}(\mathbb{R}^{n})$ and $1\leq
p\leq\infty$, it holds
\begin{equation}
\bigg\Vert \sum_{k\in\mathbb{Z}^{n}}e_{k}\otimes\phi(k)\hat{f}(k)\bigg\Vert
_{{L}^{p}{(\mathbb{T}^{n},E)}}\leq\left\Vert \mathcal{F}_{\mathbb{R}^{n}}%
^{-1}\phi\right\Vert _{{L}^{1}{(\mathbb{R}^{n})}}\,\left\Vert f\right\Vert
_{{L}^{p}{(\mathbb{T}^{n},E)}}, \label{ec5}%
\end{equation}
for all $f\in C^{\infty}(\mathbb{T}^{n},E)$.
\end{proposition}

\begin{definition}
A function $f:\mathbb{R}^{n}\longrightarrow E$ is called an $E-$
\textit{trigonom{e}tric polynomial}, if there exist $\alpha,\beta\in
\mathbb{Z}^{n}$ with $\alpha\leq\beta$ and $(x_{k})_{k\in[\alpha,\beta]}\subset E$ such that%
\begin{equation}
f=\sum_{k\in\lbrack\alpha,\beta]}e_{k}\otimes x_{k}\text{.}
\label{ec defin trigon poly}%
\end{equation}
where $e_{k}(x):=e^{ik\cdot x}$ for all $x\in\mathbb{R}^{n}$.
\end{definition}

We can write the $E-$ trigonometric polynomial $f$ in
(\ref{ec defin trigon poly}) as%
\[
f=\sum_{k\in\mathbb{Z}^{n}}e_{k}\otimes x_{k},
\]
with $x_{k}:=0$ for $k\notin\lbrack\alpha,\beta]$, or in the form
\[
f=\sum_{k\in\left[  -N,N\right]  ^{n}}e_{k}\otimes x_{k},
\]
for some $N\in\mathbb{N}$.

We denote the class of all $E-$valued trigonometric polynomials on
$\mathbb{T}^{n}$ by $\mathcal{T}(\mathbb{T}^{n},E)$. It is clear that
$\mathcal{T}(\mathbb{T}^{n},E)\subset C^{\infty}(\mathbb{T}^{n},E)$.

\begin{remark}
\label{obs2} Due to Theorem \ref{lema igualdad de distribuc} $b)$ if $f=\sum_{k\in\lbrack\alpha,\beta]}e_{k}\otimes x_{k}%
\in\mathcal{T}(\mathbb{T}^{n},X)$, then 
\[
\hat{f}(k)=\left\{
\begin{array}
[c]{ll}%
x_{k}, & \text{if }k\in\lbrack\alpha,\beta],\\
0, & \text{otherwise,}%
\end{array}
\right.
\]
when $f$ is seen as a distribution in
$\mathcal{D}^{\prime}(\mathbb{T}^{n},E)$.
\end{remark}

\begin{definition}
Let $1\leq p,q\leq\infty$, $s\in\mathbb{R}$ and $\phi:=(\phi_{j}%
)_{j\in\mathbb{N}_{0}}\in\Phi(\mathbb{R}^{n})$. We define the $E-$valued and
$n-$dimensional periodic \textit{Besov space by}%
\[
B_{p,q}^{s,\phi}(\mathbb{T}^{n},E):=\left\{  f\in\mathcal{D}^{\prime
}(\mathbb{T}^{n},E)\,:\left\Vert f\right\Vert _{B_{p,q}^{s,\phi}}:=\left\Vert f\right\Vert _{B_{p,q}^{s,\phi}%
(\mathbb{T}^{n},E)}<\infty\right\}  ,
\]
where
\begin{equation}
\left\Vert f\right\Vert _{B_{p,q}^{s,\phi}}:=%
\begin{cases}
\bigg(  \sum\limits_{j\geq0}2^{sjq}\Big\Vert \sum\limits_{k\in\mathbb{Z}^{n}%
}e_{k}\otimes\phi_{j}(k)\hat{f}(k)\Big\Vert _{{L}^{p}{(\mathbb{T}^{n},E)}%
}^{q}\bigg)  ^{1/q}, & \text{if }1\leq q<\infty,\\
& \\
\sup\limits_{j\in\mathbb{N}_{0}}2^{sj}\Big\Vert \sum\limits_{k\in
\mathbb{Z}^{n}}e_{k}\otimes\phi_{j}(k)\hat{f}(k)\Big\Vert _{{L}%
^{p}{(\mathbb{T}^{n},E)}}, & \text{if }q=\infty,
\end{cases}
\label{norma-besov}%
\end{equation}

\end{definition}

Note that for $f\in\mathcal{D}^{\prime}(\mathbb{T}^{n},E)$, $(\phi_{j}%
)_{j\in\mathbb{N}_{0}}\in\Phi(\mathbb{R}^{n})$ and $j\geq0$, $\mathcal{F}%
_{\mathbb{T}^{n}}^{-1}(\phi_{j}\mathcal{F}_{\mathbb{T}^{n}}f)\overset
{\text{(\ref{mi resultado 1})}}{=}\sum_{k\in\mathbb{Z}^{n}}e_{k}\otimes
\phi_{j}(k)\hat{f}(k)$ is a trigonometric polynomial. Furthermore
$B_{p,q}^{s,\phi}(\mathbb{T}^{n},E)$ is a Banach space with the norm defined
in (\ref{norma-besov}).

In the following theorem we prove that the $E-$valued and $n-$dimensional
periodic \textit{Besov spaces are independent on }$\phi\in\Phi(\mathbb{R}%
^{n})$.

\begin{theorem}
\label{teorema-ind-res-en-besov} Let $\phi=(\phi_{j})_{j\in\mathbb{N}_{0}}$,
$\varphi=(\varphi_{j})_{j\in\mathbb{N}_{0}}$ in $\Phi(\mathbb{R}^{n})$. Then
the norms $\left\Vert \cdot\right\Vert _{B_{p,q}^{s,\phi}(\mathbb{T}^{n},E)}$
and $\left\Vert \cdot\right\Vert _{B_{p,q}^{s,\varphi}(\mathbb{T}^{n},E)}$ are equivalent.
\end{theorem}

\begin{proof}
We must prove that there are constants $c,C>0$ such that
\begin{equation}
c\left\Vert f\right\Vert _{B_{p,q}^{s,\varphi}(\mathbb{T}^{n},E)}%
\leq\left\Vert f\right\Vert _{B_{p,q}^{s,\phi}(\mathbb{T}^{n},E)}\leq
C\left\Vert f\right\Vert _{B_{p,q}^{s,\varphi}(\mathbb{T}^{n},E)}%
,\label{ec5plus}%
\end{equation}
for all $f\in B_{p,q}^{s,\varphi}(\mathbb{T}^{n},X)$. But, due to the
transitivity of $\leq$, it suffices to show (\ref{ec5plus}) for $\phi$ as in
Lemma \ref{lema04.1}. We will show this for $1\leq p<\infty$, the case
$p=\infty$ is proved in similar way. Let $1\leq p<\infty$, $\varphi\in
\Phi(\mathbb{R}^{n})$ and $\phi$ as in Lemma \ref{lema04.1}. Since $\supp
(\phi_{j})\cap\supp(\phi_{j+2})=\emptyset$ for all $j\in\mathbb{N}_{0}$,
\begin{equation}
\phi_{j}(x)=\phi_{j}(x)\sum_{l=-1}^{1}\varphi_{j+l}(x),\qquad\forall
x\in\mathbb{R}^{n}\text{ and }j\in\mathbb{N}_{0}\text{,}\label{ec6}%
\end{equation}
where $\phi_{-1}=\varphi_{-1}:=0$. From $(|a|+|b|)^{p}\leq c_{p}%
(|a|^{p}+|b|^{p})$ for all $a,b\in\mathbb{C}$, Proposition \ref{teor4.13} and
(\ref{ec1}) it follows that
\begin{align}
\left\Vert f\right\Vert _{B_{p,q}^{s,\phi}}^{q} &
\overset{\text{(\ref{ec6})}}{=}\sum_{j\geq0}2^{jsq}\Big\Vert \sum
_{k\in\mathbb{Z}^{n}}e_{k}\otimes\phi_{j}(k)\sum_{l=-1}^{1}\varphi
_{j+l}(k)\hat{f}(k)\Big\Vert _{{L}^{p}{(\mathbb{T}^{n},E)}}^{q}\nonumber\\
&  \leq c_{q}\sum_{l=-1}^{1}\sum_{j\geq0}2^{jsq}\Big\Vert \sum_{k\in
\mathbb{Z}^{n}}e_{k}\otimes\phi_{j}(k)\varphi_{j+l}(k)\hat{f}(k)\Big\Vert
_{{L}^{p}{(\mathbb{T}^{n},E)}}^{q}\nonumber\\
&  \!\!\overset{\text{(\ref{mi resultado 01})}}{=}\!\!c_{q}\sum_{l=-1}^{1}\sum
_{j\geq0}2^{jsq}\Big\Vert \sum_{k\in\mathbb{Z}^{n}}e_{k}\otimes\phi
_{j}(k)\left(  \mathcal{F}_{{\mathbb{T}^{n}}}^{-1}\left(  \varphi
_{j+l}\mathcal{F}_{{\mathbb{T}^{n}}}f\right)  \right)  ^{\wedge}\left(
k\right)  \Big\Vert _{{L}^{p}{(\mathbb{T}^{n},E)}}^{q}\nonumber\\
&  \leq c_{q}\sum_{l=-1}^{1}\sum_{j\geq0}2^{jsq}\left\Vert \mathcal{F}%
_{{\mathbb{R}^{n}}}^{-1}\phi_{j}\right\Vert _{L^{1}({\mathbb{R}}^{n})}%
^{q}\left\Vert \mathcal{F}_{{\mathbb{T}^{n}}}^{-1}\left(  \varphi
_{j+l}\mathcal{F}_{{\mathbb{T}^{n}}}f\right)  \right\Vert _{{L}^{p}%
{(\mathbb{T}^{n},E)}}^{q}\nonumber\\
&  \leq c_{q}c_{n}\sum_{l=-1}^{1}\sum_{j\geq0}2^{jsq}\left\Vert \mathcal{F}%
_{{\mathbb{T}^{n}}}^{-1}\left(  \varphi_{j+l}\mathcal{F}_{{\mathbb{T}^{n}}%
}f\right)  \right\Vert _{{L}^{p}{(\mathbb{T}^{n},E)}}^{q}\nonumber\\
&  =c_{q}c_{n}\sum_{l=-1}^{1}\sum_{j\geq0}2^{jsq}\Big\Vert \sum
_{k\in\mathbb{Z}^{n}}e_{k}\otimes\varphi_{j+l}(k)\hat{f}(k)\Big\Vert
_{{L}^{p}{(\mathbb{T}^{n},E)}}^{q}\text{.}\label{ec antes de 7}%
\end{align}
Now, because%
\begin{multline}
\sum_{j\geq0}2^{jsq}\Big\Vert \sum_{k\in\mathbb{Z}^{n}}e_{k}\otimes
\varphi_{j\pm1}(k)\hat{f}(k)\Big\Vert _{{L}^{p}{(\mathbb{T}^{n},E)}}^{q}\\%
\leq2^{\mp jsq}\sum_{j\geq0}2^{jsq}\Big\Vert \sum_{k\in\mathbb{Z}^{n}}%
e_{k}\otimes\varphi_{j}(k)\hat{f}(k)\Big\Vert _{{L}^{p}{(\mathbb{T}^{n},E)}%
}^{q}\text{,}%
\end{multline}
then
\begin{equation}
\left\Vert f\right\Vert _{B_{p,q}^{s,\phi}({\mathbb{T}}^{n},E)}^{q}\leq
c_{q}c_{n}(2\pi)^{-n}(1+2^{sq}+2^{-sq})\left\Vert f\right\Vert _{B_{p,q}%
^{s,\varphi}({\mathbb{T}}^{n},E)}^{q}.\label{ec7}%
\end{equation}
Exchanging the roles of $\phi$ and $\varphi$ in the
expressions (\ref{ec6}) - (\ref{ec7}), we have that (\ref{ec5plus}) follows from (\ref{ec7})
with $\phi$ as in Lemma \ref{lema04.1}.
\end{proof}

Due to the last theorem we will write $B_{p,q}^{s}({\mathbb{T}}^{n},E)$
instead $B_{p,q}^{s,\phi}({\mathbb{T}}^{n},E)$. From now on, $B_{p,q}%
^{s}({\mathbb{T}}^{n},E)$ will be considered with the resolution of the unity of
Lemma \ref{lema04.1}.

\begin{remark}
Let $s\in\mathbb{R}$, $1\leq p,q\leq\infty$ and $x\in E$ fixed. Note that the
function $f:{\mathbb{T}}^{n}\rightarrow E$, defined by $f:=e_{0}\otimes x$,
with $e_{0}(y)=1$ for all $y\in\mathbb{R}^{n}$, satisfies (due to Remark
\ref{obs2})%
\begin{align}
\sum\limits_{j\geq0}2^{sjq}\Big\Vert \sum\limits_{k\in\mathbb{Z}^{n}}%
e_{k}\otimes\phi_{j}(k)\hat{f}(k)\Big\Vert _{{L}^{p}{(\mathbb{T}^{n},E)}%
}^{q}  &  =\sum\limits_{j\geq0}2^{sjq}\left\vert \phi_{j}(0)\right\vert
^{q}\left\Vert x\right\Vert _{E}^{q}\nonumber\\
&  =\left\vert \phi_{0}(0)\right\vert ^{q}\left\Vert x\right\Vert _{E}%
^{q}=:C_{0}^{q}\left\Vert x\right\Vert _{E}^{q}\text{,}
\label{ec Norm besov polino trig}%
\end{align}
if $q<\infty$. Similar result holds for $q=\infty$. Because of $(\phi
_{j})_{j\in\mathbb{N}_{0}}\in\Phi(\mathbb{R}^{n})$, one obtains with this idea
that%
\[
\mathcal{T}(\mathbb{T}^{n},E)\subset B_{p,q}^{s}({\mathbb{T}}^{n},E)\text{.}%
\]

\end{remark}
\section{Discrete Fourier multipliers}

\begin{definition}
\label{def-mult-fourier-discreto-sobre-besov} A function $M:\mathbb{Z}%
^{n}\to\mathcal{L}(E,F)$ is called a \textit{discrete 
operator-valued} ($B_{p,q}^{s}-$)\textit{ Fourier multiplier} from
$B_{p,q}^{s}(\mathbb{T}^{n},X)$ to $B_{p,q}^{s}(\mathbb{T}^{n},Y)$ if for each
$f\in B_{p,q}^{s}(\mathbb{T}^{n},E)$ there exists $g\in B_{p,q}^{s}%
(\mathbb{T}^{n},F)$ such that $\hat{g}(k)=M(k)\hat{f}(k)$ for all
$k\in\mathbb{Z}^{n}$. If $E=F$, we will say that $M$ is a discrete Fourier
multiplier on $B_{p,q}^{s}(\mathbb{T}^{n},E)$.
\end{definition}

\begin{theorem}
\label{teor criterio multiplicador}Let $M:\mathbb{Z}^{n}\longrightarrow
\mathcal{L}(E,F)$ be a function. Then the following assertion are equivalents:

\begin{itemize}
\item[a)] $M$ is a discrete $B_{p,q}^{s}-$Fourier multiplier.

\item[b)] There exists a constant $C>0$ such that%
\begin{equation}
\Big\Vert \sum_{k\in\mathbb{Z}^{n}}e_{k}\otimes M(k)\hat{f}(k)\Big\Vert
_{B_{p,q}^{s}(\mathbb{T}^{n},F)}\leq C\left\Vert f\right\Vert _{B_{p,q}^{s}(\mathbb{T}^{n},E)},
\label{criterio multiplicador}%
\end{equation}
for all $f\in B_{p,q}^{s}(\mathbb{T}^{n},E)$.
\end{itemize}
\end{theorem}

\begin{proof}
$a)\Rightarrow b)]$ Let $M:\mathbb{Z}^{n}\to\mathcal{L}(E,F)$
be a discrete $B_{p,q}^{s}-$Fourier multiplier. For $f=\sum_{k\in
\mathbb{Z}^{n}}e_{k}\otimes\hat{f}(k)\in B_{p,q}^{s}(\mathbb{T}^{n},E)$,
define
\begin{equation}
S_{M}(f):=\sum_{k\in\mathbb{Z}}e_{k}\otimes M(k)\hat{f}(k).
\end{equation}
Due to the hypothesis there exists a $g\in B_{p,q}^{s}(\mathbb{T}^{n},F)$ such
that $\hat{g}(k)=M(k)\hat{f}(k)$ far all $k\in\mathbb{Z}^{n}$. Therefore, due to (\ref{obs2.15 bienvenido}) we have%
\[
S_{M}(f)=\sum_{k\in\mathbb{Z}}e_{k}\otimes\hat{g}(k)=g,
\]
i.e. $S_{M}$ is a well defined application
from $B_{p,q}^{s}(\mathbb{T}^{n},X)$ into $B_{p,q}^{s}(\mathbb{T}^{n},F)$.
Now, we will prove that $S_{M}$ is a closed linear operator. Let
$(f_{m})_{m\in\mathbb{N}_{0}}=\left(  \sum_{k\in\mathbb{Z}^{n}}e_{k}%
\otimes\hat{f}_{m}(k)\right)  _{m\in\mathbb{N}_{0}}\subset B_{p,q}%
^{s}(\mathbb{T}^{n},E)$ such that
\[
f_{m}\xrightarrow[m\to\infty]{}f\qquad\text{ and }\qquad
S_{{M}}f_{m}\xrightarrow[m\to\infty]{}h
\]
in $B_{p,q}^{s}(\mathbb{T}^{n},E)$ and $B_{p,q}^{s}(\mathbb{T}^{n},F)$,
respectively. Since
\[
2^{sjq}\Big\Vert \sum_{k\in\mathbb{Z}^{n}}e_{k}\otimes\phi_{j}(k)(f_{m}%
-f)\,\widehat{}\,(k)\Big\Vert _{L^{p}\left(  \mathbb{T}^{n},E\right)  }%
\leq\left\Vert f_{m}-f\right\Vert _{{B_{p,q}^{s}(\mathbb{T}^{n},E)}}%
\xrightarrow[m\to\infty]{}0
\]
in $ \mathbb{C}$, then
\[
\sum_{k\in\mathbb{Z}^{n}}e_{k}\otimes\phi_{j}(k)(f_{m}-f)\,\widehat
{}\,(k)\xrightarrow[m\to\infty]{}0\quad\text{ in }%
L^{p}\left(  \mathbb{T}^{n},E\right)  ,
\]
for each $j\in\mathbb{N}_{0}$. Because of $L^{p}\left(  \mathbb{T}%
^{n},E\right)  \hookrightarrow\mathcal{D}^{\prime}\left(  \mathbb{T}%
^{n},E\right)  $, it holds for each $l\in\mathbb{Z}^{n}$ that%
\[
\phi_{j}(l)(f_{m}-f)\,\widehat{}\,(l)=\sum_{k\in\mathbb{Z}^{n}}e_{k}%
(e_{-l})\phi_{j}(k)(f_{m}-f)\,\widehat{}\,(k)\xrightarrow[m\to\infty
]{}0\quad\text{ in }E.
\]
Then
\begin{equation}
\phi_{j}(l)\hat{f}_{m}(l)\xrightarrow[m\to\infty]{}%
\phi_{j}(l)\hat{f}(l)\quad\text{ in }E,\ \forall\,l\in\mathbb{Z}^{n}\text{ and
}j\in\mathbb{N}_{0}\text{.} \label{ec1 conver una sucesion en besov}%
\end{equation}
In the same way one obtains that
\begin{equation}
\phi_{j}(l)M(l)\hat{f}_{m}(l)\xrightarrow[m\to\infty]{}\hat{h}(l)\quad\text{in }F,\ \forall\,l\in\mathbb{Z}^{n}\text{ and }%
j\in\mathbb{N}_{0}. \label{ec2 conver una sucesion en besov}%
\end{equation}
Because $M(k)\in\mathcal{L}(E,F)$, it follows from
(\ref{ec1 conver una sucesion en besov}) that for each $k\in\mathbb{Z}^{n}$
\[
\phi_{j}(k)M(k)\hat{f}_{m}(k)\xrightarrow[m\to\infty]{}\phi_{j}(k)M(k)\hat{f}(k)\quad\text{in }F.
\]
Therefore
\[
\sum_{k\in\mathbb{Z}^{n}}e_{k}\otimes\phi_{j}(k)M(k)\hat{f}_{m}(k)\xrightarrow
[m\to\infty]{F} \sum_{k\in\mathbb{Z}^{n}%
}e_{k}\otimes\phi_{j}(k)M(k)\hat{f}(k),
\]
because these sums are finite. In the same way it follows from
(\ref{ec2 conver una sucesion en besov}) that
\[
\sum_{k\in\mathbb{Z}^{n}}e_{k}\otimes\phi_{j}(k)M(k)\hat{f}_{m}(k)\xrightarrow
[m\to\infty]{F}\sum_{k\in\mathbb{Z}^{n}%
}e_{k}\otimes\phi_{j}(k)\hat{h}(k).
\]
Then
\[
\sum_{k\in\mathbb{Z}^{n}}e_{k}\otimes\phi_{j}(k)M(k)\hat{f}(k)=\sum
_{k\in\mathbb{Z}^{n}}e_{k}\otimes\phi_{j}(k)\hat{h}(k),\text{ \ \ for }%
j\in\mathbb{N}_{0}%
\]
and thus
\begin{align*}
&\left\Vert S_{{M}}f- h  \right\Vert _{{B_{p,q}^{s}(\mathbb{T}^{n},F)}}^{q}  \\
&=\sum_{j\geq0}2^{sjq}\Big\Vert \sum_{k\in\mathbb{Z}^{n}}e_{k}\otimes\phi
_{j}(k)\left(  S_{{M}}f-h\right)  ^{\wedge}(k)\Big\Vert _{{L^{p}%
(\mathbb{T}^{n};F)}}^{q}\\
&  =\sum_{j\geq0}2^{sjq}\Big\Vert \sum_{k\in\mathbb{Z}^{n}}e_{k}\otimes
\phi_{j}(k)\widehat{S_{{M}}f}(k)-\sum_{k\in\mathbb{Z}^{n}}e_{k}\otimes\phi
_{j}(k)\hat{h}(k)\Big\Vert _{{L^{p}(\mathbb{T}^{n};F)}}^{q}\\
&  =0,
\end{align*}
i.e. $S_{{M}}f=h$, and hence $S_{{M}}$ is a closed linear operator. Thus, by
the closed graph theorem, $S_{{M}}$ is bounded and consequently
(\ref{criterio multiplicador}) holds.\medskip

$b)\Rightarrow a)]$ Suppose that (\ref{criterio multiplicador}) holds for
each $f\in B_{p,q}^{s}(\mathbb{T}^{n},E)$. From this and
(\ref{ec Norm besov polino trig}) there exists a constant $c>0$ such that
\[
\left\Vert M(k)\right\Vert _{\mathcal{L}(E,F)}\leq c\text{ \ \ \ for all }%
k\in\mathbb{Z}^{n}.
\]
Let $f\in B_{p,q}^{s}(\mathbb{T}^{n},E)$. Because $f\in\mathcal{D}^{\prime
}(\mathbb{T}^{n},E)$, there exist constants $d>0$ and $N\in\mathbb{N}$ such
that
\begin{align*}
\big\Vert M(k)\hat{f}(k)\big\Vert _{{F}}  &  \leq cd\,q_{N}(e_{-k}%
)=cd\max_{\underset{|\alpha|\leq N}{\alpha\in\mathbb{N}_{0}^{n}}}\sup
_{x\in\lbrack0,2\pi]^{n}}|(-i)^{|\alpha|}k^{\alpha}e^{-ikx}|\\
&  \leq cd|k|^{N}\leq C\langle k\rangle^{N}\qquad\forall k\in\mathbb{Z}^{n}.
\end{align*}
Therefore $(M(k)\hat{f}(k))_{k\in\mathbb{Z}^{n}}\in\mathcal{O}(\mathbb{Z}^{n},F)$. 
Thus $g:=\sum_{k\in\mathbb{Z}^{n}}e_{k}\otimes
M(k)\hat{f}(k)\in\mathcal{D}^{\prime}(\mathbb{T}^{n},F)$, due to Proposition
\ref{proposicion2,12-13-barraza}, and thereby $\hat{g}(k)=M(k)\hat{f}(k)$, due
to Theorem \ref{lema igualdad de distribuc}. Furthermore, $g\in B_{p,q}%
^{s}(\mathbb{T}^{n},F)$ because of (\ref{criterio multiplicador}).
Consequently $M:\mathbb{Z}^{n}\to\mathcal{L}(E,F)$ is a
$B_{p,q}^{s}-$Fourier multiplier.
\end{proof}

\begin{remark}\ 
\begin{itemize}
\item[i)] In the proof of Theorem \ref{teor criterio multiplicador} it was shown
that $M:\mathbb{Z}^{n}\to\mathcal{L}(E)$ is a
uniformly bounded function, if $M$ is a discrete $B_{p,q}^{s}-$Fourier multiplier.

\item[ii)] If $M:\mathbb{Z}^{n}\to\mathcal{L}(E)$ is a uniformly
bounded function, then the operator $S_{M}:\mathcal{D}^{\prime}(\mathbb{T}%
^{n},E)\longrightarrow\mathcal{D}^{\prime}(\mathbb{T}^{n},E)$ defined by
\begin{equation}
S_{M}f:=\sum_{k\in\mathbb{Z}^{n}}e_{k}\otimes M(k)\hat{f}(k)
\label{ec9 defin Sm}%
\end{equation}
is well defined and
\[
(S_{M}f)\,\widehat{}\,(k)=M(k)\hat{f}(k)\text{ \ \ for all }k\in\mathbb{Z}%
^{n}\text{,}%
\]
as shown in the proof of $b)\Rightarrow a)$ in the previous theorem.
\end{itemize}
\end{remark}
A definition of $L^p$-Fourier multiplier, equivalent to the definition given in the introduction, is the following (see \cite{Na12}, Lemma 3.10):
\begin{definition}
Let $1\leq p<\infty$. A uniformly bounded function $M:\mathbb{Z}^{n}\to\mathcal{L}(E)$ 
is called a discrete $L^{p}-$ Fourier
multiplier, if there exists a constant $C>0$ such that
\begin{equation}
\left\Vert S_{M}f\right\Vert _{L^{p}\left(  \mathbb{T}^{n},E\right)  }\leq
C\left\Vert f\right\Vert _{L^{p}\left(  \mathbb{T}^{n},E\right)  }%
,\qquad\forall\,f\in\mathcal{T}(\mathbb{T}^{n},E),
\label{ec defin discreteLp multiplier}%
\end{equation}
where $S_{M}$ is defined by (\ref{ec9 defin Sm}). In this case $S_{M}%
\in\mathcal{L}(L^{p}\left(  \mathbb{T}^{n},E\right)  )$, due to the density of
$\mathcal{T}(\mathbb{T}^{n},X)$ in $L^{p}\left(  \mathbb{T}^{n},E\right)  $
(see \cite{Na12}, Proposition 2.4.). If $M$ a discrete $L^{p}-$ Fourier multiplier,
we will write $M\in\tilde{\mathcal{M}}_{p}(E)$ and $\left\Vert M\right\Vert
_{p}:=\left\Vert M\right\Vert _{\tilde{\mathcal{M}}_{p}}$ denotes the smallest
constant $C$ such that (\ref{ec defin discreteLp multiplier}) holds.
\end{definition}

Theorem 0.1 in \cite{Zim89} motivates the following definition of $UMD-$spaces.

\begin{definition}
$E$ is called a UMD$-$space, if the map $R:\mathbb{Z}^{n}\to
\mathcal{L}(E)$ defined by
\begin{equation}
R(k):=\left\{
\begin{array}
[c]{ll}%
I_{E}, & \text{if }k\geq0\text{,}\\
0, & \text{otherwise,}%
\end{array}
\right.  \label{ec proyeccion de riesz}%
\end{equation}
is a discrete operator-valued $L^{p}-$Fourier multiplier for some (or
equivalently, for all) $p\in(1,\infty)$, where $I_{E}$ is the identity
operator in $E$. We call $S_{R}$ the operator-valued $n-$dimensional
Riesz proyection.
\end{definition}

\begin{remark}
\label{obs3} It is easy to prove that $R$ is a discrete $L^{p}-$Fourier
multiplier if and only if, the map $N:\mathbb{Z}^{n}\to\mathcal{L}(E)$ defined by
\begin{equation}
N(k):=\left\{
\begin{array}
[c]{ll}%
I_{E}, & \text{if }k\leq0\text{,}\\
0, & \text{otherwise,}%
\end{array}
\right.  \label{operator-N}%
\end{equation}
is also a discrete $L^{p}$-Fourier multiplier.
\end{remark}

\begin{theorem}
\label{propiedades-multi-Fourier-discreto} Let $M,\ M_{l}\in\tilde
{\mathcal{M}}_{p}(E),\ l=1,2$, then:

\begin{itemize}
\item[a)] $M_{1}+M_{2}\in\tilde{\mathcal{M}}_{p}(E)$ with $S_{M_{1}+M_{2}%
}=S_{M_{1}}+S_{M_{2}}$.

\item[b)] $M_{1}\cdot M_{2}\in\tilde{\mathcal{M}}_{p}(E)$ with
$S_{M_{1}\cdot M_{2}}=S_{M_{1}}\circ S_{M_{2}}$, where $M_{1}\cdot M_{2}:\Z^k\to\mathcal{L}(E)$ is given by $(M_{1}\cdot M_{2})(k):=M_1(k)\circ M_2(k)$ for $k\in\Z^n$.

\item[c)] For each $\alpha\in\mathbb{Z}^{n}$ fixed, the application
$M_{\alpha}:\mathbb{Z}^{n}\to\mathcal{L}(E)$ defined by
\begin{equation}
M_{\alpha}(k):=M(k-\alpha)\quad\text{for all }k\in\mathbb{Z}^{n}, \label{ec10}%
\end{equation}
is a discrete $L_{p}-$Fourier multiplier with $\left\Vert M_{\alpha
}\right\Vert _{p}=\left\Vert M\right\Vert _{p}$.
\end{itemize}
\end{theorem}

\begin{proof}
The proof of $a)$ and $b)$ follow directly from the definition. For the proof
of $c)$ let $\alpha\in\mathbb{Z}^{n}$ fixed and $f=\sum_{k\in\mathbb{Z}^{n}%
}e_{k}\otimes x_{k}$ in $\mathcal{T}(\mathbb{T}^{n},E)$. Then
\begin{align*}
& \left\Vert   S_{M_{\alpha}}f\right\Vert _{L^{p}\left(  \mathbb{T}^{n},E\right)
}^{p} \\
  &  =\int_{\mathbb{T}^{n}}\Big\Vert \sum_{k\in\mathbb{Z}^{n}%
}e^{ik\cdot x}M(k-\alpha)x_{k}\Big\Vert _{{E}}^{p}\,\dbar x
 =\int_{\mathbb{T}^{n}}\Big\Vert \sum_{\xi\in\mathbb{Z}^{n}}e^{i(\xi
+\alpha)\cdot x}M(\xi)x_{\xi+\alpha}\Big\Vert _{{E}}^{p}\,\dbar x\\
&  \overset{y_{\xi}:=x_{\xi+\alpha}}{=}\int_{\mathbb{T}^{n}}\Big\Vert
\sum_{\xi\in\mathbb{Z}^{n}}e^{i\xi\cdot x}M(\xi)y_{\xi}\Big\Vert _{E}%
^{p}\,\dbar x
= \Big\Vert S_M\big(\sum\limits_{\xi\in\Z^n}e_\xi\otimes y_\xi\big)\Big\Vert_{L^p(\T^n,E)}^p \\
& \leq\left\Vert M\right\Vert _{p}^{p}\Big\Vert \sum_{\xi\in\mathbb{Z}^{n}%
}e_\xi\otimes y_{\xi}\Big\Vert _{L^{p}\left(  \mathbb{T}^{n},E\right)
}^{p} =\left\Vert M\right\Vert _{p}^{p}\Big\Vert \sum_{\xi+\alpha\in
\mathbb{Z}^{n}}e_{\xi+\alpha}\otimes x_{\xi+\alpha}\Big\Vert_{L^{p}\left(  \mathbb{T}^{n},E\right)  }^{p} \\
& =\left\Vert M\right\Vert _{p}%
^{p}\Big\Vert \sum_{k\in\mathbb{Z}^{n}}e_{k}\otimes x_{k}\Big\Vert_{L^{p}\left(  \mathbb{T}^{n},E\right)  }^{p}.
\end{align*}
From this follows that $M_{\alpha}$ $\in\tilde{\mathcal{M}}_{p}(E)$ with
$\left\Vert M_{\alpha}\right\Vert _{p}\leq\left\Vert M\right\Vert _{p}$. In
the same way one proves that $\left\Vert M\right\Vert _{p}\leq\left\Vert
M_{\alpha}\right\Vert _{p}$.
\end{proof}
\begin{corollary}
\label{Teorema1-caracterizacion-UMD} $E$ is a UMD-space if
and only if for each $p\in(1,\infty)$ there exists a constant $C_{p}>0$ such
that for 
$$f=\sum\limits_{k\in\left[  -K,K\right]^{n}}e_{k}\otimes x_{k}\  \in\  \mathcal{T}(\mathbb{T}^{n},E) \quad (K\in\N_0)$$ 
there exists some $\beta\in\mathbb{Z}^{n}$ which satisfies $\beta_{j}\geq K$ (for all $j=1,..,n$) and
\begin{equation}
\Big\Vert \sum_{k\in\lbrack0,\beta]}e_{k}\otimes x_{k}\Big\Vert
_{L^{p}\left(  \mathbb{T}^{n},E\right)  }\,\leq C_{p}\,\Vert f\Vert _{L^{p}\left(
\mathbb{T}^{n},E\right)  }\text{.} \label{acot de la proy de riesz}%
\end{equation}

\end{corollary}

\begin{proof}
$\Rightarrow]$ Let $E$ be a UMD$-$space, $1<p<\infty$ and 
$$f=\sum\limits_{k\in\left[  -K,K\right]^{n}}e_{k}\otimes x_{k}\  \in\  \mathcal{T}(\mathbb{T}^{n},E) \quad (K\in\N_0).$$
Then $R$ and $N$, defined as in
(\ref{ec proyeccion de riesz}) and (\ref{operator-N}), respectively, are
$L^{p}-$discrete Fourier multipliers. Due to Theorem
\ref{propiedades-multi-Fourier-discreto}$c)$,  $R_{\alpha}$ and $N_{\beta}$
are also $L^{p}-$discrete Fourier multipliers for all $\alpha,\beta\in\mathbb{Z}^{n}$. We set $x_k:=\mathbf{0}$ for $k\notin [-K,K]^n$. Then for all $\alpha,\beta\in\mathbb{Z}^{n}$ with $\alpha\leq\beta$ it holds
\begin{align}
\Big\Vert & \sum_{k\in\lbrack\alpha,\beta]}e_{k}\otimes x_{k}\Big\Vert
_{L^{p}\left(  \mathbb{T}^{n},E\right)  }   \nonumber\\
&=\Big\Vert \sum_{k\in
\mathbb{Z}^{n}}e_{k}\otimes N_{\beta}(k)R_{\alpha}(k)x_{k}\Big\Vert
_{L^{p}\left(  \mathbb{T}^{n},E\right)  }  = \Big\Vert S_{_{N_{\beta}\cdot R_{\alpha}}}\sum_{k\in\mathbb{Z}^{n}}%
e_{k}\otimes x_{k}\Big\Vert _{L^{p}\left(  \mathbb{T}^{n},E\right)
}\nonumber\\
&  \leq\left\Vert N_{\beta}\cdot R_{\alpha}\right\Vert _{p}\Big\Vert
\sum_{k\in\mathbb{Z}^{n}}e_{k}\otimes x_{k}\Big\Vert _{L^{p}\left(
\mathbb{T}^{n},E\right)  }  \leq \left\Vert N\right\Vert _{p}\left\Vert R\right\Vert _{p}\Vert
f \Vert _{L^{p}\left(
\mathbb{T}^{n},E\right)  }\text{,} \label{ec UMD implica estimat}%
\end{align}
due to Theorem \ref{propiedades-multi-Fourier-discreto}.\newline%
$\Leftarrow]$ Suppose that for $1<p<\infty$ there exists $C_p>0$ such that for each $f=\sum_{k\in\left[  -K,K\right]  ^{n}}%
e_{k}\otimes x_{k}\in\mathcal{T}(\mathbb{T}^{n},E)$ we can find
$\beta\in\mathbb{Z}^{n}$ with $\beta_{j}\geq K$ for all $j=1,\dots,n$ and such that
(\ref{acot de la proy de riesz}) holds. Then
\begin{align*}
\Vert S_Rf\Vert_{L^p(\T^n,E)} & = \Big\Vert \sum_{k\in\mathbb{Z}^{n}}e_{k}\otimes R(k)x_{k}\Big\Vert
_{L^{p}\left(  \mathbb{T}^{n},E\right)  }   = \Big\Vert \sum_{k\in
\lbrack0,K]^{n}}e_{k}\otimes x_{k}\Big\Vert _{L^{p}\left(  \mathbb{T}^{n},E\right)  }\\
&  = \Big\Vert \sum_{k\in\lbrack0,\beta]}e_{k}\otimes x_{k}\Big\Vert
_{L^{p}\left(  \mathbb{T}^{n},E\right)  }  \overset{\text{(\ref{acot de la proy de riesz})}}{\leq}C_{p}\left\Vert
f\right\Vert _{L^{p}\left(
\mathbb{T}^{n},E\right)  },
\end{align*}
for all $f=\sum_{k\in\left[  -K,K\right]  ^{n}}e_{k}\otimes x_{k}\in\mathcal{T}(\mathbb{T}^{n},E)$, and thus the operator vector-valued $n-$dimensional Riesz proyection is bounded in $L^{p}\left(  \mathbb{T}^{n},E\right)  $. Therefore $E$ is a UMD$-$space.
\end{proof}
\section{Multipliers of bounded variation; main result}

\begin{definition}
Let $G\subset\mathbb{Z}^{n}$. For a function $M:\mathbb{Z}^{n}\to{\mathcal{L}(E,F)}$ let  the \emph{restriction} of $M$ to $G$ be defined by
\[
M_{G}(k):=%
\begin{cases}
M(k), & \text{ if }k\in G,\\
0, & \text{ if }k\notin G.
\end{cases}
\]
In particular $M_{\mathbb{Z}^{n}}=M$. Let $\alpha,\beta\in\left(
\mathbb{Z\cup}\left\{  -\infty,\infty\right\}  \right)  ^{n}\ $with
$\alpha\leq\beta$. For the standar
basis of $\mathbb{R}^{n}$ $\left\{  \delta_{j}:j=1,\dots,n\right\}  $,  the difference operators $\Delta
^{\delta_{j}}$ are defined by
\[
\Delta^{\delta_{j}}M_{[\alpha,\beta]}(x):=\left\{
\begin{array}
[c]{ll}M_{[\alpha,\beta]}(x)-M_{[\alpha,\beta]}(x-\delta_{j}), & \text{if }x_{j}\neq\alpha_{j},\\
0\text{,} & \text{if }x_{j}=\alpha_{j}.
\end{array}
\right.
\]
Moreover, let $\Delta^{0}M_{[\alpha,\beta]}:=M_{[\alpha,\beta]}$,%
\[
\Delta^{\gamma}M_{[\alpha,\beta]}:=\Delta^{\gamma_{1}\delta_{1}}
\cdots\Delta^{\gamma_{n}\delta_{n}}M_{[\alpha,\beta]}\text{,}%
\qquad\text{for }\gamma=(\gamma_{1},\dots,\gamma_{n})\in\{1,0\}^{n}\text{,}%
\]
and let the \emph{variation} of $M$ on $[\alpha,\beta]$ be defined by 
\begin{equation}
\var\limits_{[\alpha,\beta]}M_{{[\alpha,\beta]}}:=\sum_{\xi\in
\lbrack\alpha,\beta]}\left\Vert \Delta^{\gamma_{\xi}}M_{[\alpha,\beta]}%
(\xi)\right\Vert,\label{variation de M}%
\end{equation}
where
$\gamma_{\xi}=(\gamma_{\xi_{1}},\dots,\gamma_{\xi_{n}})$ with
\begin{equation}
\gamma_{\xi_{j}}:=\left\{
\begin{array}
[c]{ll}%
1\text{,} & \text{if }\xi_{j}\neq\alpha_{j},\\
0\text{,} & \text{if }\xi_{j}=\alpha_{j}.
\end{array}
\right.  \label{ec-ast}%
\end{equation}
\end{definition}
Note that if $\alpha,\beta\in\mathbb{Z}^{n}$ with $\alpha\leq\beta$ and $M:\mathbb{Z}^{n}\to{\mathcal{L}(E,F)}$ is a function such that $M_{\left[\alpha,\beta\right]  }=0$ in $\mathbb{Z}^{n}$, then $\Delta^{\delta_{j}}M_{[\alpha,\beta]}(k)=0$ for all $k\in\mathbb{Z}^{n}$ and $j=1,\dots,n$, and in consequence $ \var\limits_{[\alpha,\beta]}M_{[\alpha,\beta]} = 0$.
\begin{remark}
\label{teorema4.42}Using properties of telescopic sums it can be seen that for
each $\beta\in\mathbb{Z}^{n}$ it holds
\begin{equation}
M(\beta)=\sum_{\xi\in\lbrack\alpha,\beta]}\Delta^{\gamma_{\xi}}M_{[\alpha
,\beta]}(\xi) \label{ec principal-1}%
\end{equation}
for all $\alpha\in\mathbb{Z}^{n}$ with $\alpha\leq\beta$.
\end{remark}

\begin{definition}
The coarse decomposition of $\mathbb{Z}^{n}$ is defined by: $D_{0}:=\{0\}$ and
for $d\in\mathbb{N}$,
\begin{multline*}
D_{d} :=\big\{  k\in\mathbb{Z}^{n}:|k_{1}|,\dots,|k_{l-1}|<2^{r+1},\ 2^{r}%
\leq|k_{l}|<2^{r+1},\\
 |k_{l+1}|,\dots,|k_{n}|<2^{r}\big\}  ,
\end{multline*}
where $d=nr+l$ with $r\in\mathbb{N}_{0}$ and $l\in\{1,2,\dots,n\}$. For $d\in\mathbb{N}$, $D_{d}=D_{d^{+}}\cup D_{d^{-}}$ where $D_{d^{\pm}}:=\left\{  k\in D_{d}\,:\,\pm k_{l}>0\right\}  $. Furthermore
\begin{align*}
\var\limits_{D_d}M & :=\var\limits_{D_{d^{+}}}M_{D_{d^{+}}} + \var\limits_{D_{d^{-}}}M_{D_{d^{-}}}\quad (d\in\N)\quad \text{and}\\
\var\limits_{D_{0}}M & :=\var\limits_{D_{0}}M_{D_{0}}.
\end{align*}
\end{definition}

Note that for $d\in\mathbb{N}$ and $D_{d^{\pm}}$ as in the above definition,
$D_{d^{\pm}}=\left[  \alpha_{d^{\pm}},\beta_{d^{\pm}}\right]  $ for
some $\alpha_{d^{\pm}},\beta_{d^{\pm}}\in\mathbb{Z}^{n}$ (for
example, $\alpha_{d^{+}}=\left(  -2^{r+1},...,-2^{r+1},2^{r},-2^{r},...,-2^{r}\right)$,
where $2^{r}$ is in the $l-$th position). Therefore $\var\limits_{D_{d^{\pm}}}M_{D_{d^{\pm}}}$ make sense.\\
\\
Now, the variational Marcinkiewicz condition, given in \cite{AB04} by%
\[
\sup_{k\in\mathbb{Z}}\left\Vert M_{k}\right\Vert +\sup_{j\geq0}\sum
\limits_{2^{j}\leq\left\vert k\right\vert \leq2^{j+1}}\left\Vert M_{k+1}%
-M_{k}\right\Vert <\infty\text{,}%
\]
will be generalised by the following definition.

\begin{definition}
Let $M:\mathbb{Z}^{n}\to{\mathcal{L}(E,F)}$ be uniformly
bounded. $M$ is called a function of bounded variation with respect to the coarse
decomposition of $\mathbb{Z}^{n}$, if there exists a positiv constant $C$ such
that
\begin{equation}\label{generalizacion de la condicion variacional de Marcinkiewics}%
\sup_{d\in\mathbb{N}_{0}}\var\limits_{D_{d}}M<C\text{.}
\end{equation}
\end{definition}

\begin{lemma}
\label{ejemplo-def-M} Let $j\in\mathbb{N}_{0}$. Then the function
$M:=M_j:\mathbb{Z}^{n}\to{\mathcal{L}(E)}$ defined by
\begin{equation}\label{def-M}
M(k):=\begin{cases}
I_E, & \text{if}\ k=k_{1}\delta_{1}\ \text{with}\ k_{1}\in\left[7\cdot2^{j-3},2^{j}\right] \ \text{and}\ j\geq3,\\
0, & \text{otherwise}
\end{cases}
\end{equation}
is of bounded variation with respect to the coarse decomposition of
$\mathbb{Z}^{n}$.
\end{lemma}

\begin{proof}
By definition $M:\mathbb{Z}^{n}\to{\mathcal{L}(E)}$ satisfies
$\left\Vert M(k)x\right\Vert \leq\left\Vert x\right\Vert $ for all
$k\in\mathbb{Z}^{n}$ and $x\in E$. Therefore $\left\{  M(k):k\in\mathbb{Z}%
^{n}\right\}  \subset{\mathcal{L}(E)}$ is uniformly bounded with
\begin{equation}
\left\Vert M(k)\right\Vert _{{\mathcal{L}(E)}}\leq1\quad\text{for all }%
k\in\mathbb{Z}^{n}. \label{ec41}%
\end{equation}
We will show that $M$ satisfies
(\ref{generalizacion de la condicion variacional de Marcinkiewics}). In fact, if $M_{D_{0}}=0$, then $\var\limits_{D_{0}}M=0$. Now, we fix $d\in\mathbb{N}$ with $d=nj+l$, $j\in\mathbb{N}_{0}$ and $l\in\left\{1,\dots,n\right\}$. Due to \eqref{def-M} we have:
\begin{itemize}
\item[i)] If $j<3$, then $M_{D_{d}}=0$ and therefore $\var\limits_{D_{d}}M=0$.

\item[ii)] From $j\geq3$ and $l\in\{2,\dots,n\}$ it follows $M_{D_{d}}=0$ because if $k\in D_{d}$, then $k_{l}\neq0$
and hence $k\neq k_{1}\delta_{1}$ which yields $M_{D_{d}}(k)=0$. Moreover, by
definition $M_{D_{d}}(k)=0$ if $k\notin D_{d}$. Therefore $\var\limits_{D_{d}}M=0$.

\item[iii)] If $j\geq3$ and $l=1$, then for each $k\in\mathbb{Z}^{n}$ it holds that
\[
M_{D_{d}}(k)=\begin{cases}
I_{E}, & \text{if }k=2^{j}\delta_{1},\\
0, & \text{otherwise.}
\end{cases}
\]
It follows that $M_{D_{d^{-}}}=M_{{D_{d^{+}}\setminus\{2^{j}\delta_{j}\}}}=0$,
and then
\[
\var\limits_{D_{d}}M=\sum_{k\in D_{d^{+}}}\left\Vert \Delta^{\gamma_{k}%
}M_{D_{d^{+}}}(k)\right\Vert _{{\mathcal{L}(E)}}=\left\Vert \Delta
^{\gamma_{2^{j}\delta_{1}}}M(2^{j}\delta_{1})\right\Vert _{{\mathcal{L}(E)}}.
\]
Since $D_{d^{+}}=[\alpha_{d^{+}},\beta_{d^{+}}]$ with $\alpha_{d^{+}}%
=(2^{j},-2^{j},\dots,-2^{j})$, we have
\begin{align*}
&\var\limits_{D_{d}}M\\
  &  =\left\Vert \Delta^{\gamma_{2^{j}\delta{_{1}}}}M_{D_{d^{+}}}(2^{j}\delta_{1})\right\Vert _{{\mathcal{L}(E)}}
 =\big\Vert \Delta^{0}\Delta^{\delta_{2}}\cdots\Delta^{\delta_{n}}M_{D_{d^{+}}}(2^{j}\delta_{1})\big\Vert _{{\mathcal{L}(E)}}\\
&  =\big\Vert \Delta^{\delta_{2}}\cdots\Delta^{\delta_{n}}M_{D_{d^{+}}}(2^{j}\delta_{1})\big\Vert _{{\mathcal{L}(E)}}\\
&  =\big\Vert  \Delta^{\delta_{2}}\cdots\Delta^{\delta
_{n-1}} \big(  M_{D_{d^{+}}}(2^{j}\delta_{1})-\underbrace
{M_{D_{d^{+}}}(2^{j}\delta_{1}-\delta_{n})}_{=0}\big)  \big\Vert
_{{\mathcal{L}(E)}}\\
&  =\big\Vert \Delta^{\delta_{2}}\cdots\Delta^{\delta_{n-1}}M_{D_{d^{+}}}(2^{j}\delta_{1})\big\Vert _{{\mathcal{L}(E)}} = \cdots 
 =\left\Vert M(2^{j}\delta_{1})\right\Vert _{{\mathcal{L}(E)}}\leq1,
\end{align*}
due to (\ref{ec41}). 
\end{itemize}
In consequence $M:\mathbb{Z}^{n}\to{\mathcal{L}(E)}$ defined by (\ref{def-M}) is of bounded variation with
respect to the coarse decomposition of $\mathbb{Z}^{n}$.
\end{proof}

\begin{lemma}
\label{de la contenencia del soporte} Let $(\phi_{j})_{j\geq0}$ be as in Lemma
\ref{lema04.1}.

\begin{itemize}
\item[a)] For $j\geq1$ it holds
\begin{equation}
\supp(\phi_{j})\cap\mathbb{Z}^{n}\,\subset\,\bigcup\limits_{d=n(j-m-1)+1}^{nj+n}D_{d} \ \,=:\,\ D_{j}^{n}\text{,}%
\end{equation}
where $D_{-d}:=D_{0}$ for $d\in\mathbb{N}$ and $m$ is the smallest
non-negative integer satisfying $\sqrt{n}\leq2^{m}$.

\item[b)] $\supp(\phi_{0})\cap\mathbb{Z}^{n}\subset\bigcup\limits_{d=0}%
^{n}D_{d}$.
\end{itemize}
\end{lemma}

\begin{proof}
a) Let $k\in\mathbb{Z}^{n}\cap\supp(\phi_{j})$ with $j\geq1$, then
$2^{j-1}<|k|<2^{j+1}$. If $k\notin D_{j}^{n}$, $k\in D_{d}$ with $d=nr+l$ for
some $l\in\{1,\dots,n\}$ and some $r\geq j+1$ or $r\leq j-m-2$. In the first
case it holds
\[
|k|\geq|k_{l}|\geq2^{r}\geq2^{j+1},
\]
which contradicts that $|k|<2^{j+1}$. Now, we consider the second case, i.e.
$k\in D_{d}$ with $d=nr+l$ for some $l\in\{1,\dots,n\}$ and some $r\leq
j-m-2$. If $r\geq0$, then $|k_{s}|<2^{r+1}\leq2^{j-m-1}$ for all
$s\in\{1,\dots,n\}$, and thus
\[
|k|\leq\sqrt{n}|k|_{\infty}\leq\sqrt{n}2^{j-m-1}\leq2^{j-1},
\]
which now is in contradiction with $|k|>2^{j-1}$. The same happens when $r<0$ since
$D_{-d}=D_{0}$. In consequence $k\in D_{d}$ with $d=nr+l$ for some
$l\in\{1,\dots,n\}$ and some $r\in\{j-m-1,\dots,j\}$.\\
b) If $0\neq k=(k_{1},\dots,k_{n})\in\supp(\phi_{0})\cap\mathbb{Z}^{n}$,
then $|k_{s}|<2$ for all $s\in\{1,2,\dots,n\}$ and therefore $k\in D_{l}$ for some
$l\in\{1,2,\dots,n\}$, since otherwise there would be some $r\in\mathbb{N}$
and $s\in\{1,2,\dots,n\}$ such that $|k_{s}|\geq2^r$. Then we have that
\[
\lbrack\text{supp}(\phi_{0})\cap\mathbb{Z}^{n}]\!\setminus\!\{0\}\subset
\bigcup_{d=1}^{n}D_{d}.
\]
From this follows b), due to $0\in D_{0}$.
\end{proof}

Now, we will prove the main result of this paper. But before note that
\begin{equation}
\sum_{k\in\lbrack\alpha,\beta]}\sum_{l\in\lbrack\alpha,k]}a_{l}b_{k}%
=\sum_{k\in\lbrack\alpha,\beta]}a_{k}\sum_{l\in\lbrack k,\beta]}b_{l}\text{,}
\label{ec principal-2}%
\end{equation}
for all $\alpha,\beta\in\mathbb{Z}^{n}$ with $\alpha\leq\beta$.

\begin{theorem}
\label{RESULTADO-PRINCIPAL} Let $s\in\mathbb{R}$, $1<p<\infty$ and $1\leq q\leq\infty$.
Each function $M:\mathbb{Z}^{n}\to\mathcal{L}\left(  E\right)  $
of bounded variation with respect to the coarse decomposition of
$\mathbb{Z}^{n}$ is a Fourier multiplier on $B_{p,q}^{s}(\mathbb{T}^{n},E)$ if
and only if $E$ is a UMD-space.
\end{theorem}

\begin{proof}
$\Leftarrow]$ Let $E$ be a UMD-space. Suppose that
$M:\mathbb{Z}^{n}\to\mathcal{L}\left(  E\right)  $ satisfies
(\ref{generalizacion de la condicion variacional de Marcinkiewics}), $f\in B_{p,q}^{s}(\mathbb{T}^{n},E)$ and let
$(\phi_{j})_{j\geq0}$ be as in Lemma \ref{lema04.1}. Due to Lemma
\ref{de la contenencia del soporte} we obtain that for $j\geq1$ and
$x\in\mathbb{T}^{n}$ fixed it holds%
\begin{align}
&  \Big\Vert \sum_{k\in\mathbb{Z}^{n}}e^{ik\cdot x}\phi_{j}(k)M(k)\hat
{f}(k)\Big\Vert _{{E}} = \Big\Vert \sum_{k\in\text{supp}(\phi_{j})\cap\mathbb{Z}^{n}}e^{ik\cdot
x}\phi_{j}(k)M(k)\hat{f}(k)\Big\Vert _{{E}}\nonumber\\
& \leq \sum_{d=n(j-m-1)+1}^{n(j+1)}\Big\Vert \sum_{k\in D_{d}}e^{ik\cdot x}%
\phi_{j}(k)M(k)\hat{f}(k)\Big\Vert _{{E}}\nonumber\\
&  \leq\sum_{d=n(j-m-1)+1}^{n(j+1)}\bigg(  \Big\Vert \sum_{k\in D_{d^{+}}%
}e^{ik\cdot x}\phi_{j}(k)M(k)\hat{f}(k)\Big\Vert _{{E}}\nonumber\\
& \qquad\qquad\qquad\qquad\qquad\qquad + \quad\Big\Vert\sum_{k\in D_{d^{-}}}e^{ik\cdot x}\phi_{j}(k)M(k)\hat{f}(k)\Big\Vert
_{E}\bigg). \label{ec 1 prueba teo principal}%
\end{align}
Now we consider the sum over $D_{d^{+}}:=[\alpha_{_{d^{+}}},\beta_{_{d^{+}}}]$.%
\begin{align*}
&  \Big\Vert \sum_{k\in D_{d^{+}}}e^{ik\cdot x}\phi_{j}(k)M(k)\hat
{f}(k)\Big\Vert _{{E}}\\
&  \!\!\overset{\text{\eqref{ec principal-1}}}{=}\Big\Vert \sum_{k\in D_{d^{+}}}\sum_{\xi\in\lbrack\alpha_{d^{+}},k]}\Delta^{\gamma_{\xi}}M_{[\alpha_{d^{+}},k]}(\xi)e^{ik\cdot x}\phi_{j}(k)\hat{f}(k)\Big\Vert _{{E}}\\
&  =\Big\Vert \sum_{k\in D_{d^{+}}}\sum_{\xi\in\lbrack\alpha_{d^{+}},k]}\Delta^{\gamma_{\xi}}M_{[\alpha_{d^{+}},\beta_{d^{+}}]}(\xi)e^{ik\cdot
x}\phi_{j}(k)\hat{f}(k)\Big\Vert _{{E}}\\
&  \!\!\overset{\text{\eqref{ec principal-2}}}{=}\Big\Vert \sum_{k\in D_{d^{+}}}\Delta^{\gamma_{k}}M_{[\alpha_{d^{+}},\beta_{d^{+}}]}(k)\sum_{\xi\in\lbrack
k,\beta_{d^{+}}]}e^{i\xi\cdot x}\phi_{j}(\xi)\hat{f}(\xi)\Big\Vert _{{E}}\\
&  \leq\sup_{k\in D_{d^{+}}}\sum_{k\in
D_{d^{+}}}\left\Vert \Delta^{\gamma_{k}}M_{[\alpha_{d^{+}},\beta_{d^{+}}]}(k)\right\Vert _{{\mathcal{L}\left(  E\right)  }}\Big\Vert \sum_{\xi\in\lbrack k,\beta_{d^{+}}]}e^{i\xi\cdot x}\phi_{j}(\xi)\hat{f}(\xi)\Big\Vert _{{E}}\\
&  \!\!\overset
{\text{(\ref{generalizacion de la condicion variacional de Marcinkiewics})}}{\leq}C\sup_{k\in D_{d^{+}}}\Big\Vert \sum_{\xi\in\lbrack k,\beta_{d^{+}}]}e^{i\xi\cdot x}\phi_{j}(\xi)\hat{f}(\xi)\Big\Vert _{{E}} \overset{\text{(\ref{ec UMD implica estimat})}}{\leq}K_{p}\Big\Vert
\sum_{k\in\mathbb{Z}^{n}}e^{ik\cdot x}\phi_{j}(k)\hat{f}(k)\Big\Vert _{{E}}\text{.}%
\end{align*}
We get the same estimate for the sum over $D_{d^{-}}$ with a similar procedure. Then, from
\eqref{ec 1 prueba teo principal} it follows that
\[
\Big\Vert \sum_{k\in\mathbb{Z}^{n}}e^{ik\cdot x}\phi_{j}(k)M(k)\hat
{f}(k)\Big\Vert _{{E}}\leq2K_{p}n(m+2)\Big\Vert \sum_{k\in\mathbb{Z}^{n}%
}e^{ik\cdot x}\phi_{j}(k)\hat{f}(k)\Big\Vert _{{E}}\text{.}%
\]
Analogously, using Lemma \ref{de la contenencia del soporte} $b)$  we obtain
\[
\Big\Vert \sum_{k\in\mathbb{Z}^{n}}e^{ik\cdot x}\phi_{0}(k)M(k)\hat
{f}(k)\Big\Vert _{E}\leq2K_{p}n\Big\Vert \sum_{k\in\mathbb{Z}^{n}}e^{ik\cdot x}\phi_{0}(k)\hat{f}(k)\Big\Vert _{{E}}\text{.}%
\]
Thus, there exists a constant $C>0$ such that
\[
\Big\Vert \sum_{k\in\mathbb{Z}^{n}}e_{k}\otimes\phi_{j}(k)M(k)\hat
{f}(k)\Big\Vert _{L^{p}\left(  \mathbb{T}^{n}.E\right)  }\leq C\Big\Vert
\sum_{k\in\mathbb{Z}^{n}}e_{k}\otimes\phi_{j}(k)\hat{f}(k)\Big\Vert
_{L^{p}\left(  \mathbb{T}^{n}.E\right)  }%
\]
for all $f\in B_{p,q}^{s}(\mathbb{T}^{n},E)$ and $j\in\mathbb{N}_{0}$, and
therefore
\[
\Big\Vert \sum_{k\in\mathbb{Z}^{n}}e_{k}\otimes M(k)\hat{f}(k)\Big\Vert
_{B_{p,q}^{s}(\mathbb{T}^{n},E)}\leq C\left\Vert f\right\Vert _{B_{p,q}^{s}(\mathbb{T}^{n},E)}.
\]
Thus, Theorem \ref{teor criterio multiplicador} implies that $M$ is a 
$B_{p,q}^{s}(\T^n,E)$-Fourier multiplier.

$\Rightarrow]$ Now, we suppose that each function $M:\mathbb{Z}^{n}\to\mathcal{L}\left(  E\right)  $ satisfying
\eqref{generalizacion de la condicion variacional de Marcinkiewics} is a
 $B_{p,q}^{s}(\T^n,E)$-Fourier multiplier. Let $(\phi_{\ell})_{\ell\in
\mathbb{N}_{0}}\in\Phi(\mathbb{R}^{n})$ be as in Lemma \ref{lema04.1} and fix $j\in\N$ with $j\geq3$. For this $j$ let
$M:\mathbb{Z}^{n}\to\mathcal{L}\left(  E\right)  $ be the function
given in Lemma \ref{ejemplo-def-M}. Moreover, let us consider an
arbitrary sequence $\left(  x_{k}\right)  _{k\in\mathbb{Z}^{n}}$ in $E$ and
the $E$-valued trigonometric polynomial
\[
h:=\sum_{\substack{k=k_{1}\delta_{1},\\7\cdot2^{j-3}\leq k_{1}\leq
3\cdot2^{j-1}}}e_{k}\otimes x_{k}.%
\]
This $h$ can be written as
\[
h=\sum_{\substack{k=k_{1}\delta_{1},\\7\cdot2^{j-3}\leq k_{1}\leq3\cdot
2^{j-1}}}e_{k}\otimes\hat{h}(k)\text{,}%
\]
where $\hat{h}(k)=0$ for $k\notin\left\{  k_{1}\delta_{1}\,:\,7\cdot
2^{j-3}\leq k_{1}\leq3\cdot2^{j-1}\right\}  $ and $\hat{h}(k)=x_{k}$ else, due
to Remark \ref{obs2}. By Lemma \ref{lema04.1}, $\phi_{j}(x)=1$ for all
$x\in\mathbb{R}^{n}$ with $7\cdot2^{j-3}\leq|x|\leq3\cdot2^{j-1}$ and  $\phi_{l}(x)=0$ for all $x\in\mathbb{R}^{n}$ with $7\cdot2^{j-3}%
\leq|x|\leq3\cdot2^{j-1}$ and $l\neq j$. Thus
\begin{align}
 \Big\Vert \sum_{k\in\mathbb{Z}^{n}}&\!e_{k}\otimes M(k)\hat{h}(k)\Big\Vert
_{{B_{p,q}^{s}(\mathbb{T}^{n},E)}}^{q}  \nonumber\\
& =\sum_{l\geq0}2^{qsl}\bigg\Vert
\sum_{\substack{k=k_{1}\delta_{1},\\7\cdot2^{j-3}\leq k_{1}\leq3\cdot2^{j-1}}}\!e_{k}\otimes\phi_{l}(k)M(k)\hat{h}(k)\bigg\Vert _{L^{p}\left(
\mathbb{T}^{n}.E\right)  }^{q}\nonumber\\
&  \!\!\!\overset{\text{(\ref{def-M})}}{=}\sum_{l\geq0}2^{qsl}\bigg\Vert
\sum_{\substack{k=k_{1}\delta_{1},\\7\cdot2^{j-3}\leq k_{1}\leq2^{j}}}e_{k}\otimes\phi_{l}(k)\hat{h}(k)\bigg\Vert _{L^{p}\left(  \mathbb{T}^{n}.E\right)  }^{q} \nonumber\\
&= 2^{qsj}\bigg\Vert \sum_{\substack{k=k_{1}\delta_{1}\\7\cdot2^{j-3}\leq
k_{1}\leq2^{j}}}e_{k}\otimes\hat{h}(k)\bigg\Vert _{L^{p}\left(\mathbb{T}^{n}.E\right)  }^{q}. \label{4.49}%
\end{align}
Similarly we obtain that
\begin{align}
\left\Vert h\right\Vert _{B_{p,q}^{s}(\mathbb{T}^{n},E)}^{q}  &  =\sum
_{l\geq0}2^{qsl}\bigg\Vert \sum_{\substack{k=k_{1}\delta_{1},\\7\cdot
2^{j-3}\leq k_{1}\leq3\cdot2^{j-1}}}e_{k}\otimes\phi_{l}(k)\hat{h}%
(k)\bigg\Vert _{L^{p}\left(  \mathbb{T}^{n}.E\right)  }^{q}\nonumber\\
&  =2^{qsj}\bigg\Vert \sum_{\substack{k=k_{1}\delta_{1},\\7\cdot2^{j-3}\leq
k_{1}\leq3\cdot2^{j-1}}}e_{k}\otimes\hat{h}(k)\bigg\Vert _{L_{p}(\mathbb{T}^{n},E)}^{q}. \label{4.50}%
\end{align}
From Theorem \ref{teor criterio multiplicador}, \eqref{4.49} and \eqref{4.50}
it follows that
\begin{equation}
\bigg\Vert \sum_{\substack{k=k_{1}\delta_{1},\\7\cdot2^{j-3}\leq k_{1}\leq2^{j}}}e_{k}\otimes x_{k}\bigg\Vert _{L^{p}\left(  \mathbb{T}^{n}.E\right)  }\leq C\bigg\Vert \sum_{\substack{k=k_{1}\delta_{1}%
,\\7\cdot2^{j-3}\leq k_{1}\leq3\cdot2^{j-1}}}e_{k}\otimes x_{k}\bigg\Vert
_{L^{p}\left(  \mathbb{T}^{n}.E\right)  }. \label{ec4.4-1}%
\end{equation}
From \eqref{ec4.4-1} we can write%
\begin{equation*}
\Big\Vert \sum_{\ell=7\cdot2^{j-3}}^{2^{j}}e_{\ell}\otimes x_{\ell}\Big\Vert
_{L^{p}\left(  \mathbb{T},E\right)  }\leq C_{n}\Big\Vert \sum_{\ell=7\cdot
2^{j-3}}^{3\cdot2^{j-1}}e_{\ell}\otimes x_{\ell}\Big\Vert _{L^{p}\left(
\mathbb{T},E\right)  } \label{ec4.4-2}%
\end{equation*}
for all $\left(  x_{\ell}\right)  _{\ell\in\mathbb{N}_{0}}\subset E$ and therefore
\begin{equation}
\Big\Vert \sum_{k\in\left[  0,2^{j-3}\right]  }e_{k}\otimes x_{k}\Big\Vert
_{L^{p}(\mathbb{T},E)}\leq C_{n}\Big\Vert \sum_{k\in\left[  -2^{j-3},2^{j-1}\right]  }e_{k}\otimes x_{k}\Big\Vert _{L^{p}\left(  \mathbb{T},E\right)  } \label{ne5}%
\end{equation}
for all $\left(  x_{k}\right)  _{k\in\mathbb{N}_{0}}\subset E$.\\
\\
Now, let $f=\sum\limits_{k\in\left[  -N,N\right]  }e_{k}\otimes x_{k}\in\mathcal{T}(\mathbb{T},E)$ and set $x_k=0$ for $k\notin [-N,N]$. There exists some $j_{N}\geq3$ such that
$N\leq2^{j_{N}-3}$ and
\begin{align*}
\Big\Vert \sum_{k\in\left[  0,2^{j_{N}-3}\right]  }e_{k}\otimes
x_{k}\Big\Vert _{L^{p}\left(  \mathbb{T}.E\right)  } & \overset
{\text{(\ref{ne5})}}{\leq}C_{n}\Big\Vert \sum_{k\in\left[  -2^{j_{N}%
-3},2^{j_{N}-1}\right]  }e_{k}\otimes x_{k}\Big\Vert _{L^{p}\left(
\mathbb{T},.E\right)  }\\
& \!\!\quad = \  C_{n}\,\Vert f \Vert _{L^{p}\left(  \mathbb{T},E\right)  }\text{.}%
\end{align*}
Therefore $E$ is a $UMD$-space due to Corollary
\ref{Teorema1-caracterizacion-UMD}.
\end{proof}

\begin{remark}
\label{Remark desigualdad multiplier vs symbol}In the proof of Theorem
\ref{RESULTADO-PRINCIPAL} we have proved that if $E$ is a $UMD-$space,
$s\in\mathbb{R}$, $1<p<\infty$, $1\leq q\leq\infty$ and $M$ is of bounded
variation with respect to the coarse decomposition of $\mathbb{Z}^{n}$, then
there exists $C>0$ such that%
\[
\left\Vert S_{M}\right\Vert _{{\mathcal{L}}\left(  B_{p,q}^{s}(\mathbb{T}^{n},E)\right)  }\leq C\sup\limits_{d\in\mathbb{N}_{0}}\var\limits_{D_{d}}\,M\text{.}%
\]

\end{remark}

\begin{remark}
\label{Remark bounded variation}Let $M:\mathbb{Z}^{n}\longrightarrow
{\mathcal{L}(E,F)}$ be uniformly bounded.

\begin{itemize}
\item[a)] As a particular case of the proof of Theorem 3.24 a) in
\cite{Na12}, it holds that $M$ is of bounded variation with $\sup\limits_{d\in
\mathbb{N}_{0}}\var\limits_{D_{d}}\,M\leq 2^{3n+1}$, if the set%
\[
\left\{  \left\vert k\right\vert ^{\left\vert \gamma_{k}\right\vert }\Delta
M_{D_{d}}\left(  k\right)  :d\in\mathbb{N}_{0}\text{ and }k\in D_{d}\right\}
\]
is uniformly bounded.

\item[b)] It is easy to see that%
\begin{align*}
\big\{  \left\vert k\right\vert ^{\left\vert \gamma_{k}\right\vert }\Delta
M_{D_{d}} & \left(  k\right)\, :  \, d\in\mathbb{N}_{0}\text{ and }k\in D_{d}\big\}\\
& \subset\big\{  \left\vert k\right\vert ^{\left\vert \gamma\right\vert }\Delta
M\left(  k\right)  :k\in\mathbb{Z}^{n}\text{ and }\gamma\in\left\{
0,1\right\}  ^{n}\big\}  \text{.}%
\end{align*}

\end{itemize}
\end{remark}

\section{Periodic boundary valued problems}

In this section we will study  the existence and uniqueness of solution for the
problems \eqref{Prob1} and \eqref{Prob2}. Note that $A(t)$ given in \eqref{Glei differoperat} is a (formal) lineal differential operator with $\mathcal{L}\left(  E\right)
-$valued coefficients, where $a^{0}:\left[  0,\infty\right)  \times
\mathbb{R}^{n}\rightarrow\mathcal{L}\left(  E\right)  $,%
\begin{equation}
a^{0}\left(  t,\xi\right)  :=\sum_{\left\vert \alpha\right\vert =m}a_{\alpha
}\left(  t\right)  \xi^{\alpha}\label{Gleichung hauptsymbol}%
\end{equation}
is called its \emph{principal symbol}.

For $\theta\in\left[  0,\pi\right]  $, set $\sum_{\theta}:=\left\{  \lambda
\in\mathbb{C}:\left\vert \arg\lambda\right\vert \leq\theta\right\}
\cup\left\{  0\right\}  $. Given $\kappa\geq1$ and $\theta\in\left[
0,\pi\right)  $, the operator $A$ is called (uniformly\textbf{)} $\left(
\kappa,\theta\right)  -$elliptic if $\sum_{\theta}\subset\rho\left(
-a^{0}\left(  t,\xi\right)  \right)  $ and%
\begin{equation}
\Big\Vert \left[  \lambda I+a^{0}\left(  t,\xi\right)  \right]
^{-1}\Big\Vert _{\mathcal{L}\left(  E\right)  }\leq\frac{\kappa
}{1+\left\vert \lambda\right\vert }\quad\text{ for all }\lambda\in\sum
\nolimits_{\theta}\label{Ungleich (k,teta)-elliptisch}%
\end{equation}
and $\left(  t,\xi\right)  \in\left[  0,\infty\right)  \times\mathbb{R}^{n}$
with $\left\vert \xi\right\vert =1$. It is called $\theta-$elliptic, if it
is\textbf{ }$\left(  \kappa,\theta\right)  -$elliptic for some $\kappa\geq1$,
and normally-elliptic if it is $\frac{\pi}{2}-$elliptic.

\begin{remark}[{\cite{Am01}, Remarks 3.1}]\label{Bem. eigenschaft ellipt differ}\ 
\begin{itemize} 
\item[a)] Condition
($\ref{Ungleich (k,teta)-elliptisch}$) is equivalent to%
\[
\Big\Vert \left[  \lambda I+a^{0}\left(  t,\xi\right)  \right]
^{-1}\Big\Vert _{\mathcal{L}\left(  E\right)  }\leq\frac{\kappa}{\left\vert
\xi\right\vert ^{m}+\left\vert \lambda\right\vert }%
\]
for all $\lambda\in\sum_{\theta}$ and $\left(  t,\xi\right)  \in\left[
0,\infty\right)  \times\mathbb{R}^{n}$ with $\xi\neq0$.

\item[b)] The order $m$ is even whenever $A$ is normally elliptic.
\end{itemize}
\end{remark}

\begin{remark}
\label{Remark estimate for elliptic op}Let $A$ be uniformly $\left(
\kappa,\theta\right)  -$elliptic, $a\left(  t,\xi\right)  :=\sum_{\left\vert
\alpha\right\vert \leq m}a_{\alpha}\left(  t\right)  \xi^{\alpha}$ and
$b:=a-a^{0}$. Due to%
\[
\lambda I+a\left(  t,\xi\right)  =\left[  I+b\left(  t,\xi\right)  \left(
\lambda I+a^{0}\left(  t,\xi\right)  \right)  ^{-1}\right]  \left(  \lambda
I+a^{0}\left(  t,\xi\right)  \right)  \text{,}%
\]
by a Neumann series argument, there exists some $\omega_{0}>0$ such that%
\begin{equation}
\Big\Vert \left[  \lambda I+a\left(  t,\xi\right)  \right]  ^{-1}\Big\Vert
_{\mathcal{L}\left(  E\right)  }\leq\frac{2\kappa}{\left\vert \xi\right\vert
^{m}+\left\vert \lambda\right\vert } \label{ec estimate for elliptic op}%
\end{equation}
for all $\lambda\in\omega_{0}+\sum_{\theta}$ and $\left(  t,\xi\right)
\in\left[  0,\infty\right)  \times\mathbb{R}^{n}$.
\end{remark}

\begin{proposition}
\label{Prop for example}Let $s\in\mathbb{R}$, $1<p<\infty$, $1\leq q\leq
\infty$, $E$  a $UMD-$space, $A$  an uniformly $\left(  \kappa
_{0},\theta\right)$-elliptic differential operator satisfaying
$\sum_{\left\vert \alpha\right\vert \leq m}\left\Vert a_{\alpha}\right\Vert
_{\infty}\leq C$, and let%
\[
\mathcal{A}:=\mathcal{A}_{B_{p,q}^{s}}:B_{p,q}^{s+m}(\mathbb{T}^{n}%
,E)\rightarrow B_{p,q}^{s}(\mathbb{T}^{n},E),\quad u\longmapsto Au,
\]
be the $B_{p,q}^{s}-$realization of $A$. Then there exist \ $\kappa\geq1$ and
$\omega_{0}>0$ such that $\omega_{0}+\sum_{\theta}\subset\rho\left(
-\mathcal{A}\left(  t\right)  \right)  $ and%
\begin{equation}
\Big\Vert \left(  \lambda I+\mathcal{A}\left(  t\right)  \right)
^{-1}\Big\Vert _{\mathcal{L}\left(  B_{p,q}^{s}(\mathbb{T}^{n},E)\right)
}\leq\frac{\kappa}{1+\left\vert \lambda\right\vert }
\label{ec genera semigroup on Besov}%
\end{equation}
for all $\lambda\in\omega_{0}+\sum\nolimits_{\theta}$  and $t\geq0$. 
In particular, each $\mathcal{A}\left(  t\right)  $ generates an analytic
semigroup on $B_{p,q}^{s}(\mathbb{T}^{n},E)$, if $A$ is uniformly normally elliptic.
\end{proposition}

For the proof of this proposition we will use the following notations and
lemma, whose proof can be found in \cite{Na12}.

Giving $\alpha\in\mathbb{N}_{0}^{n}\backslash\left\{  0\right\}  $, let%
\[
\mathcal{Z}_{\alpha}:=\Big\{  \mathcal{W=}\left(  w^{1},...,w^{r}\right)
:1\leq r\leq\left\vert \alpha\right\vert ,0<w^{j}\leq\alpha,\sum_{j=1}%
^{r}w^{j}=\alpha\Big\}
\]
denote the set of all additive decompositions of $\alpha$ into
$r=r_{\mathcal{W}}$ multi-indices. For the sake of consistence we set
$\mathcal{Z}_{0}:=\left\{  \emptyset\right\}  $ and $r_{\emptyset}:=0$. For
$\mathcal{W=}\left(  w^{1},...,w^{r}\right)  \in\mathcal{Z}_{\alpha}$ let
$w_{\ast}^{j}$ be defined by%
\[
w_{\ast}^{j}:=\sum_{l=j+1}^{r}w^{l}\text{.}%
\]

\begin{lemma}[\cite{Na12}, Lemma 7.1c]\label{Lemma 7.1 Nau12}
Let $S:\mathbb{Z}^{n}\to\mathcal{L}\left(  E,F\right)  $ be a function such that the
inverse $\left(  S^{-1}\right)  (k):=\left(  S(k)\right)  ^{-1}$ exists for
all $k\in\mathbb{Z}^{n}$. Then for $\alpha\in\mathbb{N}_{0}^{n}$, we have%
\[
\Delta^{\alpha}\left(  S^{-1}\right)  (k)=\sum_{\mathcal{W}\in\mathcal{Z}%
_{\alpha}}\left(  -1\right)  ^{r_{\mathcal{W}}}\left(  S^{-1}\right)
(k-\alpha)
{\displaystyle\prod\limits_{j=1}^{r_{\mathcal{W}}}}
\Big(  \big(  \Delta^{w^{j}}S\big)  S^{-1}\Big)(k-w_{\ast}%
^{j})%
\]
for $k\in\mathbb{Z}^{n}$.
\end{lemma}

\begin{proof}
[Proof of Proposition \ref{Prop for example}]Let $a$ be as in Remark
\ref{Remark estimate for elliptic op} and $\gamma\in\left\{  0,1\right\}
^{n}$. For $\lambda\in\omega_{0}+\sum_{\theta}$ and $t\geq0$, we define
$M_{\lambda,t}\left(  k\right)  :=\lambda\left(  \lambda+a\left(
t,\cdot\right)  \right)  ^{-1}\left(  k\right)  $, $k\in\mathbb{Z}^{n}$. Using
Lemma \ref{Lemma 7.1 Nau12}, the triangular inequality, the fact that
$\Delta^{w^{j}}\big(  k-w_{\ast}^{j}\big)^{\alpha}$ is a polynomial in
$k-w_{\ast}^{j}$ of degree not greater than $\vert \alpha\vert -\vert
w^{j}\vert $ and (\ref{ec estimate for elliptic op}), we obtain for all
$k\in\mathbb{Z}^{n}$ that%
\begin{align*}
&  \left\vert k\right\vert ^{\left\vert \gamma\right\vert }\left\Vert
\Delta^{\gamma}M_{\lambda,t}\left(  k\right)  \right\Vert \\
&  \leq\left\vert \lambda\right\vert \left\vert k\right\vert ^{\left\vert
\gamma\right\vert }\sum_{\mathcal{W}\in\mathcal{Z}_{\gamma}}\big\Vert \left(
\lambda+a\left(  t,k-\gamma\right)  \right)  ^{-1}\big\Vert\cdot\\
&\qquad \qquad \qquad \qquad  \cdot{\displaystyle\prod\limits_{j=1}^{r_{\mathcal{W}}}}
\big(  \big\Vert \Delta^{w^{j}}a\left(  t,k-w_{\ast}^{j}\right)  \big\Vert
\big\Vert \left(  \lambda+a\left(  t,k-w_{\ast}^{j}\right)  \right)
^{-1}\big\Vert \big) \\
&  \leq C\left\vert \lambda\right\vert \left\vert k\right\vert ^{\left\vert
\gamma\right\vert }\sum_{\mathcal{W}\in\mathcal{Z}_{\gamma}}\big\Vert \left(
\lambda+a\left(  t,k-\gamma\right)  \right)^{-1}\big\Vert\cdot \\
& \qquad \qquad \qquad \qquad \cdot {\displaystyle\prod\limits_{j=1}^{r_{\mathcal{W}}}}
\Big(  \sum_{\left\vert \alpha\right\vert \leq m}\big\vert \Delta^{w^{j}}\left(  k-w_{\ast}^{j}\right)^{\alpha}\big\vert \big\Vert \left(
\lambda+a\left(  t,k-w_{\ast}^{j}\right)  \right)  ^{-1}\big\Vert \Big) \\
&  \leq C2\kappa\left\vert k\right\vert ^{\left\vert \gamma\right\vert }%
\sum_{\mathcal{W}\in\mathcal{Z}_{\gamma}}\frac{\left\vert \lambda\right\vert}{\left\vert k-\gamma\right\vert ^{m}+\left\vert \lambda\right\vert }\cdot\\
& \qquad \qquad \qquad \qquad \cdot{\displaystyle\prod\limits_{j=1}^{r_{\mathcal{W}}}}
\Big(  \sum_{\left\vert \alpha\right\vert \leq m}\sum\limits_{\text{finite}}c_{\alpha,w^{j}}\vert k-w_\ast^j\vert ^{m-\vert w^{j}\vert
}\frac{2\kappa}{\vert k-w_\ast^j\vert ^{m}+\vert \lambda\vert
}\Big) \\
&  \leq C_{\kappa}\left\vert k\right\vert ^{\left\vert \gamma\right\vert }%
\sum_{\mathcal{W}\in\mathcal{Z}_{\gamma}}C_{\mathcal{W}}\frac{\vert\lambda\vert }{\vert k-\gamma\vert ^{m}+\vert \lambda\vert }%
{\displaystyle\prod\limits_{j=1}^{r_{\mathcal{W}}}}
\Big(  \vert k-w_\ast^j\vert ^{-\vert w^{j}\vert }%
\frac{\vert k-w_\ast^j\vert^{m}}{\vert k-w_\ast^j\vert^{m}+\vert\lambda\vert }\Big) \\
&  \leq C_{\kappa}\left\vert k\right\vert ^{\left\vert \gamma\right\vert }%
\sum_{\mathcal{W}\in\mathcal{Z}_{\gamma}}C_{\mathcal{W}}%
{\displaystyle\prod\limits_{j=1}^{r_{\mathcal{W}}}}
\vert k-w_\ast^j\vert ^{-\vert w^{j}\vert }\leq\widehat
{C}_{\kappa}\text{,}%
\end{align*}
where $\widehat{C}_{\kappa}$ is a constant which do not depend on $\lambda$ and $t$, and $\Vert \cdot\Vert$ abbreviates $\Vert \cdot\Vert_{\mathcal{L}(E)}$. It follows that $M_{\lambda,t}$ is of bounded variation due to Remark
\ref{Remark bounded variation}. Thus Theorem \ref{RESULTADO-PRINCIPAL} implies
that $M_{\lambda,t}$ is a discrete Fourier multiplier on $B_{p,q}%
^{s}(\mathbb{T}^{n},E)$ and (\ref{ec genera semigroup on Besov}) holds due to
Remark \ref{Remark desigualdad multiplier vs symbol}.
\end{proof}

\begin{corollary}\label{Corollary-application}
Let $0<\rho<1$, $s\in\mathbb{R}$, $1<p<\infty$, $1\leq q\leq\infty$, $E$
a $UMD$-space and $A$ a uniformly normally elliptic differential
operator satisfaying%
\begin{equation}
\left(  t\mapsto a_{\alpha}(t)\right)  \in C^{\rho}\left(  \left[
0,T\right]  ,\mathcal{L}\left(  E\right)  \right)
\label{ec condition symbol hoelder}%
\end{equation}
for all $\left\vert \alpha\right\vert \leq m$.

\begin{itemize}
\item[a)] If $f\in C^{\rho}\left(  \left[  0,T\right]  ,B_{p,q}^{s}\left(
\mathbb{T}^{n},E\right)  \right)  $, then the problem \eqref{Prob1} has a
unique classical solution%
\[
u\in C^{\rho}\left(  \left(  0,T\right]  ,B_{p,q}^{m+s}\left(  \mathbb{T}%
^{n},E\right)  \right)  \cap C^{1+\rho}\left(  \left(  0,T\right]
,B_{p,q}^{s}\left(  \mathbb{T}^{n},E\right)  \right)  \text{.}%
\]

\item[b)] If $s_{1}\in\mathbb{R}$, $1\leq p_{1},q_{1}\leq\infty$ and
$\omega_{0}$ as in Proposition \ref{Prop for example}, then for each $f\in
B_{p_{1},q_{1}}^{s_{1}}\left(  \mathbb{T},B_{p,q}^{s}\left(  \mathbb{T}%
^{n},E\right)  \right)  $ and $\omega\geq\omega_{0}$ there exists a unique
$u\in B_{p_{1},q_{1}}^{1+s_{1}}\left(  \mathbb{T},B_{p,q}^{s}\left(
\mathbb{T}^{n},E\right)  \right)  $ such that $u(t)+A_{\omega}u(t)=f(t)$ for
almost all $t\in\left[  0,2\pi\right]  $. In this sense $u$ is the unique
solution for the problem \eqref{Prob2}.
\end{itemize}
\end{corollary}

\begin{proof}\ 
\begin{itemize}
\item[a)] This is a consequence of Proposition \ref{Prop for example},
\eqref{ec condition symbol hoelder}, Theorems 1.2 and 1.3 in \cite{Ta60} and
S\"{a}tzes 4.11 and 4.12 in \cite{Pou65}.
\item[b)] This follows from Proposition \ref{Prop for example} and Theorem
5.1 in \cite{AB04}.
\end{itemize}
\end{proof}

\bibliographystyle{amsalpha}

\bigskip

\end{document}